\documentclass{amsart}
\usepackage{tikz}
\usepackage{enumerate}
\usepackage[pdftex]{hyperref}
\DeclareMathOperator{\Id}{\; Id}
\newtheorem{theo}{Theorem}

\newtheorem{prop}[theo]{Proposition}
\newtheorem{cor}[theo]{Corollary}
\newtheorem{defi}[theo]{Definition}

\newtheorem{rem}[theo]{Remark}

\newtheorem{exa}[theo]{Example}

\numberwithin{theo}{section}
\numberwithin{equation}{section}

\DeclareMathOperator{\Real}{Re}

\def\NN{{\mathbb N}}
\def\RR{{\mathbb R}}
\def\CC{{\mathbb C}}

\def\cd{{\hbox{\rm card}}}

\def\DD{{{\mathcal{D}}}}

\def\MM{{{\mathcal{M}}}}

\def\G{{{\Gamma}}}
\def\a{{\alpha}}
\def\b{{\beta}}

\def\g{{\gamma}}

\def\l{{\lambda}}

\def\p{{\pi}}

\def\f{{\varphi}}

\def\CCC{{\mathcal C}}

\def\setprop{{\, \vrule width 0.6pt height 8pt depth 2pt \;}}

\def\AD{{\mathcal A}} 

\def\ZD{{\mathcal Z}}

\def\uu{{\bf U}}

\def\be{{\,\bf e}}
\def\Ker{{\rm Ker\;}}

\title[Spectral analysis of the second derivative]{The spectrum of the Hilbert space valued second derivative\\ with general self-adjoint boundary conditions}
\author{Joachim von Below}
\address{Joachim von Below, LMPA Joseph Liouville ULCO, FR CNRS Math. 2956,
Universit\'e Lille Nord de France,
50 rue F. Buisson, B.P. 699, F--62228 Calais Cedex, France}
\email{joachim.von.below@lmpa.univ-littoral.fr}
\author{Delio Mugnolo}
\address{Delio Mugnolo, Institut f\"ur Analysis, Universit\"at Ulm, Helmholtzstra{\ss}e 18, 89081 Ulm, Germany}
\email{delio.mugnolo@uni-ulm.de}
\thanks{This article was partially completed during several visits by the second author at the ULCO in Calais. He is grateful to the ULCO for its financial support, and to the first author for his warm hospitality. The first author is grateful to the University of Ulm
for several invitations during the last years. This research has also been partially supported by the Land Baden--W\"urttemberg in the framework of the \emph{Juniorprofessorenprogramm} -- research project on ``Symmetry methods in quantum graphs''. \\ {\indent\textbf{Corresponding author}: Joachim von Below,
LMPA Joseph Liouville ULCO, FR CNRS Math. 2956, Universit\'e Lille Nord de France, 50 rue F. Buisson, B.P. 699, F--62228 Calais Cedex, France; email: joachim.von.below@lmpa.univ-littoral.fr;  Tel.: +33.3.21.46.36.27; Fax: +33.3.21.46.36.69}}

\begin{document}

\begin{abstract}

We consider a large class of self-adjoint elliptic problem associated with the second derivative acting on a space of vector-valued functions. We present two different approaches to the study of the associated eigenvalues problems. The first, more general one allows to replace a secular equation (which is well-known in some special cases) by an abstract rank condition. The latter seems to apply particularly well to a specific boundary condition, sometimes dubbed ``anti-Kirchhoff'' in the literature, that arise in the theory of differential operators on graphs; it also permits to discuss interesting and more direct connections between the spectrum of the differential operator and some graph theoretical quantities. In either case our results yield, among other, some results on the symmetry of the spectrum.
\end{abstract}
\maketitle

\section{Introduction}\label{intro}

Differential equations on networks have a long history, starting probably in 1847 with Kirchhoff's electrical circuit equations using
the potential mesh rule and an incident currency law \cite{kirch}.  For specific differential edge operators both latter conditions were naturally replaced 
 were by the continuity potential condition and an incident flow condition. Nowadays, dealing with rather general vertex transition  condition, we are led to abstract interacting problems, as e.g.\ the Bochner space setting below.

The problem of determining all possible self-adjoint realizations of a differential operator is quite common in mathematical physics, as it is always possible conceivably to prepare a corresponding quantum system in such a way that the associated Hamiltonian is one such realization.

Relying upon known results for differential operators on domains, in~\cite{KosSch99} V.\ Kostrykin and R.\ Schrader have proposed a natural representation of self-adjoint boundary conditions for 1-dimensional systems, i.e., for Laplace operators on a Bochner space $L^2(0,1;\ell^2(E))$. If $E$ is the edge set of a graph, then such systems are usually called ``networks'' or ``quantum graphs'' in the literature. That approach has the drawbacks that it is poorly fitted for a variational setting, and that the boundary conditions do not determine their representations uniquely. Both issues can be avoided making use of an alternative parametrization proposed by P.\ Kuchment in~\cite{Kuc04}. Both parametrization are equivalent, in the sense that each Kuchment's boundary condition can be represented using Kostrykin--Schrader's formalism, and vice versa. Moreover, if $E$ is finite, i.e., if one considers a differential operator on finitely many edges, then it has been shown in~\cite{Kuc04} that self-adjoint boundary conditions are actually exhausted by either parametrization. 
Kuchment's parametrization can be reduced to the choice of a closed subspace $Y$ of $\ell^2(E)$ and of a bounded linear self-adjoint $R$ on $\ell^2(E)$; see Section~\ref{eigen1} for details. To the best of our knowledge, not much is known about spectral properties of general self-adjoint realizations, beside the resolvent formula of Krein type obtained in~\cite{Pan06,Pos07}. The goal of our note is to determine the (real) spectrum of this class of Hamiltonians in dependence of $Y$ and $R$.

It turns out that, on a spectral level, a different choice of $R$ has the sole effect of shifting (in a complicated way) the eigenvalues, but the essential choice is, in a certain sense, that of $Y$. This becomes particularly clear at the end of Section~\ref{eigen1}, where, after obtaining a general description of the spectrum in the general case, we present a particularly appealing spectral symmetry if $R=0$. Moreover, we show how our results can be further simplified in the case of small sets $E$ -- and in particular in the case of $\ell^2(E)=\mathbb C$,i.e., of a single loop.

There are at least two special cases where our abstract secular equation has a much more intuitive meaning: the boundary conditions may reflect so-called \emph{continuity} and \emph{Kirchhoff conditions} -- this is in fact by far the most common condition considered in the literature, cf.\ Section~\ref{adjcal}. In Section~\ref{sec:antikir} we consider a further -- less standard but interesting -- boundary condition, which is in some sense dual to the previous one: the function satisfies Kirchhoff's law where its associated flow is continuous in the network's vertices.

As a motivation for the study of this special and seemingly counter-intuitive boundary condition, it should be remarked that this so-called ``anti-Kirchhoff'' condition is already known to be distinguished among all possible self-adjoint conditions for the second derivative on a graph. E.g., it has been observed in~\cite{ExnTur06,Mug10} that, apart from the decoupled conditions of Dirichlet/Neumann type, Kirchhoff and anti-Kirchhoff conditions are the sole ones invariant under edge permutations (up to lower order perturbations); and on a loop, also the sole ones that induce $L^\infty$-contractive heat semigroups.

In Sections~\ref{adjcal} and~\ref{sec:antikir} we study these two boundary conditions thoroughly. In this special case, the general secular equation turns out to be tightly related to the characteristic polynomial of the transition matrix of the underlying abstract graph. We can then strongly improve our general results and are in fact able to describe the spectrum of the Laplacian under said conditions in terms of invariants of the underlying graph: 
several features are proved to reflect directly known graph theoretical properties and quantities, like bipartiteness.  Our results can be seen as partial solutions to certain inverse problems. In Section \ref{sec:inverse} we compare information made available by our results with the outcome of previous linear algebraic investigations.

\section{Eigenvalue problems for the second derivative on general Bochner spaces}\label{eigen1}

Let $H$ be a separable Hilbert space and consider the Bochner space $L^2(0,1;H)$. Let $Y$ be a closed subspace of $H\times H$ and consider the Helmholtz equation with boundary conditions
\begin{equation}\label{bc1}
\begin{pmatrix}u(0)\\u(1)\end{pmatrix}\in Y\hbox{ and }\begin{pmatrix}-u'(0)\\u'(1)\end{pmatrix}\in Y^\perp,
\end{equation}
i.e., the eigenvalue problem
\begin{equation}\tag{$\rm{EP_Y}$}
\left\{
\begin{array}{rcll}
-u''(x)&=&\lambda u(x),\qquad &x\in(0,1)\\
\begin{pmatrix}u(0)\\u(1)\end{pmatrix}&\in& Y,\\
\begin{pmatrix}-u'(0)\\u'(1)\end{pmatrix}&\in& Y^\perp,\\
\end{array}
\right.
\end{equation}
in $L^2(0,1;H)$. This kind of boundary conditions might be generalized considering
\begin{equation}\tag{$\rm{EP_{Y,R}}$}
\left\{
\begin{array}{rcll}
u''(x)&=&\lambda u(x),\qquad &x\in(0,1)\\
\begin{pmatrix}u(0)\\u(1)\end{pmatrix}&\in& Y,\\
\begin{pmatrix}-u'(0)\\u'(1)\end{pmatrix}+R\begin{pmatrix}u(0)\\u(1)\end{pmatrix}&\in& Y^\perp,\\
\end{array}
\right.
\end{equation}
where $R\in{\mathcal L}(Y)$. This class of boundary conditions has been made popular by~\cite{Kuc04}, but we point out that they, and even some generalization to eigenvalue-dependent boundary conditions, had already been suggested in~\cite[\S3]{be2}.

\begin{rem}
In the special setting of diffusion on $1$-dimensional ramified structures (so-called \emph{networks} or \emph{quantum graphs}) this class of boundary conditions has been proposed in~\cite{Kuc04}. Observe that the boundary conditions in $\rm(EP_{Y,R})$ are less general than those of Agmon--Douglis--Nirenberg type
\begin{equation}\tag{ADN}
A\begin{pmatrix}u(0)\\u(1)\end{pmatrix}+B\begin{pmatrix}-u'(0)\\u'(1)\end{pmatrix}=0
\end{equation}
for $A,B\in{\mathcal L}(H\times H)$: this is due to the fact that the first boundary condition in $\rm(EP_{Y,R})$ determines the second completely. E.g., if we take $H={\mathbb C}$, then taking 
$$Y=\left\{ \begin{pmatrix}c\\ c\end{pmatrix}\setprop c\in\mathbb C\right\},$$
i.e., letting 
$$f(0)=f(1)$$ 
forces the second boundary condition to read
$$f'(0)=f'(1)+Rf(0)$$
for some $R\in \mathbb C$. Thus, if we consider
$$A:=\begin{pmatrix}
1 & -1\\ 0 & 0
\end{pmatrix}\qquad \hbox{and}\qquad B:=\begin{pmatrix}
0 & 0 \\ 2 & -1
\end{pmatrix}
,$$
thus obtaining continuity
$$f(0)=f(1)$$ 
and a second boundary condition
$$2f'(0)=f'(1)$$
 other than Kirchhoff, these both boundary conditions cannot be expressed in the form of $(\rm{EP_{Y,R}})$ for any choice of $R$.
\end{rem}

In spite of their less general character, the boundary conditions in~$\rm(EP_{Y,R})$ give the problem an interesting variational structure that will prove extremely useful in spectral investigations.
The sesquilinear form associated with $(\rm{EP_{Y,R}})$ is given by
$$a_R(f,g):=\int_0^1 (f'(x)| g'(x))_{H} dx+\left( R\begin{pmatrix}
f(0)\\ f(1)
\end{pmatrix}\Big| \begin{pmatrix}
g(0)\\ g(1)
\end{pmatrix}\right)_{H\times H}$$
with dense domain 
$$V_Y:=\left\{f\in H^1(0,1;H) \setprop \begin{pmatrix}f(0)\\f(1)\end{pmatrix}\in Y\right\}.$$ 

We denote by $-\Delta_{Y,R}$ the operator associated with $(a_R,V_Y)$, i.e.,
\begin{eqnarray*}
D(\Delta_{Y,R})&:=&\{f\in V_Y\setprop\exists g\in L^2(0,1;H) \hbox{ s.t. }a_R(f,h)=(g|h)_H\hbox{ for all }h\in V_Y\}\\
&=&\left\{f\in H^2(0,1;H)\setprop\begin{pmatrix}
-f'(0)\\ f'(1)
\end{pmatrix}+R\begin{pmatrix}
f(0)\\ f(1)
\end{pmatrix}\in Y^\perp \hbox{ and }\begin{pmatrix}
f(0)\\ f(1)
\end{pmatrix}\in Y\right\},\\
\Delta_{Y,R}f&:=&-g.
\end{eqnarray*}

The model is still sufficiently easy to allow us to explicitly compute the eigenvalues of the corresponding problem $\rm(EP_{Y,R})$. We mention that, with a different representation of the boundary conditions, the number of negative eigenvalues has been computed in~\cite{BehLug10}.

\begin{prop}\label{multipl0}
Let $Y$ be a closed subspace of $H\times H$ and $R\in {\mathcal L}(Y)$ be self-adjoint. The spectrum of the operator $-\Delta_{Y,R}$ is real; it is contained in $[0,\infty)$ if $R$ is positive semidefinite. If $H$ is finite-dimensional, then $-\Delta_{Y,R}$ has pure point spectrum.
Then the following assertions hold.

1)  $0$ is an eigenvalue of $\rm(EP_{Y,R})$ if and only if 
\[
H_Y:=\{(A,B)\in H\times H:A=B\}\cap Y\not=\{0\}\quad \hbox{and}\quad R=0,
\]
and in this case the multiplicity of $0$ agrees with $
\dim H_Y\le \dim H.$

2) Let $\lambda>0$. If $R$ is positive semidefinite, then  $\lambda >0$ is an eigenvalue of $\rm(EP_{Y,R})$ if and only if the space $H_{Y,R}$ of all solutions $(A,B)\in H\times H$ of the system 
\begin{equation}
\label{PYrankgen}
\left\{
\begin{array}{l}
P_{Y^\perp}\begin{pmatrix}
A\\ A\cos\sqrt{\lambda}+B\sin\sqrt{\lambda}
\end{pmatrix}=0\\
P_{Y}\left(\sqrt{\lambda }\begin{pmatrix}
B\\ A\sin\sqrt{\lambda}-B\cos\sqrt{\lambda}
\end{pmatrix}-R\begin{pmatrix}
A\\ A\cos\sqrt{\lambda}+B\sin\sqrt{\lambda}
\end{pmatrix}\right)=0
\end{array}
\right.
\end{equation}
has nonzero dimension; and in this case the multiplicity of $\lambda$ agrees with $\dim H_{Y,R}\le 2\dim H$. 
\end{prop}

\begin{proof}
The fact that the form $(a_R,V_Y)$ is real and, for positive semidefinite $R$, accretive shows that the spectrum of $-\Delta_{Y,R}$ is contained in the positive half line. If $H$ is finite-dimensional, then the compact embedding of $D(\Delta_{Y,R})$ into $L^2(0,1;H)$ follows from the Aubin--Lions Lemma and yields that $-\Delta_{Y,R}$ has pure point spectrum.

(1) If
$$u''(x)=0,\qquad x\in [0,1],$$
then 
$$0=a(u,u)=\int_0^1 \|u'(x)\|^2_H dx+\left( R\begin{pmatrix}
u(0)\\ u(1)
\end{pmatrix}\Big| \begin{pmatrix}
u(0)\\ u(1)
\end{pmatrix}\right)_{H\times H}\ge \int_0^1 \|u'(x)\|^2_H dx,$$
hence $u$ can only be an eigenvector for the eigenvalue $0$ if it is entry-wise constant, i.e.,
$$u(x)\equiv A\in H\qquad \hbox{for all }x\in [0,1].$$
Hence, in order for $u$ to satisfy the first boundary condition in $\rm(EP_{Y,R})$ (the second one being satisfied if and only if $R=0$ or $Y=\{0\}$, of course) one needs that 
$$\begin{pmatrix}
u(0)\\ u(1)
\end{pmatrix}\equiv 
\begin{pmatrix}
A \\ A
\end{pmatrix}\in Y.$$

(2) For $\lambda>0$, we see that the characteristic equation 
$$-u''(x)=\lambda u(x),\qquad x\in [0,1],$$
is solved for all $\lambda>0$ by 
$$u(x):=A\cos \sqrt{\lambda }x +B\sin \sqrt{\lambda }x,\qquad x\in [0,1],$$
for some $A,B\in H$, whence
$$u'(x)=\sqrt{\lambda }(-A\sin \sqrt{\lambda }x +B\cos \sqrt{\lambda }x),\qquad x\in [0,1],$$
and hence
$$\begin{pmatrix}
u(0)\\ u(1)
\end{pmatrix}=
\begin{pmatrix}
A\\
A\cos \sqrt{\lambda } +B\sin \sqrt{\lambda }
\end{pmatrix}$$
along with
$$\begin{pmatrix}
-u'(0)\\ u'(1)
\end{pmatrix}=\sqrt{\lambda}
\begin{pmatrix}
-B\\
-A\sin \sqrt{\lambda } +B\cos \sqrt{\lambda }
\end{pmatrix}.$$
Imposing the boundary conditions~\eqref{bc1} we obtain
$$P_{Y^\perp}\begin{pmatrix}
A\\ A\cos\sqrt{\lambda}+B\sin\sqrt{\lambda}
\end{pmatrix}=0
$$
and
$$
P_{Y}\left(\sqrt{\lambda }\begin{pmatrix}
B\\ A\sin\sqrt{\lambda}-B\cos\sqrt{\lambda}
\end{pmatrix}-R\begin{pmatrix}
A\\ A\cos\sqrt{\lambda}+B\sin\sqrt{\lambda}
\end{pmatrix}\right)=0.$$
Hence, $\lambda>0$ is an eigenvalue if and only if the above equations have a nontrivial solution $(A,B)\in H\times H$.

Combining the cases in (1) and (2) yields the claimed rank condition. Finally, the bound on the dimension of the eigenspace associated with $0$ follows from Grassmann's formula, as
\[
\dim H_Y=\dim H+\dim Y-\dim(\{(A,B)\in H\times H\setprop A=B \}+ Y)\le \dim H.
\]
This concludes the proof.  
\end{proof}

If $R=0$, an easy but interesting consequence of this general result is the following.
\begin{cor}\label{multipl0cor}
Let $Y$ be a closed subspace of $H\times H$ and $R=0$. Let both $\lambda>0$ and $(\pi-\sqrt{\lambda})^2>0$. Then $\lambda$ is an eigenvalue of $\rm(EP_{Y})$ if and only if $(\pi-\sqrt{\lambda})^2$ is an eigenvalue of $\rm(EP_{Y^\perp})$. In this case, their algebraic multiplicities coincide.
\end{cor}
\begin{proof}
By Proposition~\ref{multipl0}, $\lambda$ is an eigenvalue of $\rm(EP_{Y})$ if and only if the space $H_{Y}$ of all solutions $(A,B)\in H\times H$ of the system 
\begin{equation}
\label{PYrankgennoB}
\left\{
\begin{array}{l}
P_{Y^\perp}\begin{pmatrix}
A\\ A\cos\sqrt{\lambda}+B\sin\sqrt{\lambda}
\end{pmatrix}=0\\
P_{Y}\begin{pmatrix}
B\\ A\sin\sqrt{\lambda}-B\cos\sqrt{\lambda}
\end{pmatrix}=0
\end{array}
\right.
\end{equation}
has dimension larger than $0$. Now, it suffices to observe that because
\[
\sin\sqrt{(\pi-\sqrt{\lambda})^2}=\sin\sqrt{\lambda}\qquad \hbox{and}\qquad	\cos\sqrt{(\pi-\sqrt{\lambda})^2}=-\cos\sqrt{\lambda},
\]
$(A,B)$ is a solution of the system in~\eqref{PYrankgennoB} if and only if $(\tilde{A},\tilde{B}):=(B,A)$ is a solution of the system in~\eqref{PYrankgennoB} with the roles of $Y,Y^\perp$  interchanged -- i.e., of 
\begin{equation}
\label{PYrankgennoB2}
\left\{
\begin{array}{l}
P_{Y}\begin{pmatrix}
\tilde{A}\\ \tilde{A}\cos\sqrt{\lambda}+\tilde{B}\sin\sqrt{\lambda}
\end{pmatrix}=0\\
P_{Y^\perp}\begin{pmatrix}
\tilde{B}\\ \tilde{A}\sin\sqrt{\lambda}-\tilde{B}\cos\sqrt{\lambda}
\end{pmatrix}=0
\end{array}
\right.
\end{equation}
This concludes the proof.
\end{proof}

\begin{exa}\label{Ykirch-comp}
In the special case where $\rm(EP_{Y,R})$ can be represented as an eigenvalue problem on a network associated with a directed graph $\Gamma$, the most natural boundary conditions are those requiring continuity of the function values as well as a Kirchhoff-like rule for the normal derivatives at each vertex of $\Gamma$, see Section~\ref{adjcal} for details. Such conditions have been discussed in the literature since~\cite{Lum80}, and in fact implicitly even since~\cite{RueSch53}. It is well-known that the associated elliptic problem is self-adjoint, and in fact we can represent them in our formalism letting 
$$Y:={\rm Range}\,\tilde{ \DD},$$ 
where 
\begin{equation}\label{deftilde}
\tilde{ \DD}:=\begin{pmatrix} ({ \DD}^+)^T \\({ \DD}^-)^T\end{pmatrix},
\end{equation}
and ${ \DD}^+,{ \DD}^-$ are the matrices whose entries are the positive and negative parts of the entries of the  (signed) incidence matrix $\DD$ of $\Gamma$. In this case one can check that the dimension of $Y$ agrees with the number of vertices, i.e., with the cardinality of $V$, but for general $Y$ there is no such an intuitive interpretation. An exception is constituted by the ``dual'' conditions to the above one, i.e., by the conditions arising from the choice 
$$Y:=\left({\rm Range}\,\tilde{ \DD}\right)^\perp,$$
which in the literature are usually referred to as \emph{$\delta'$-couplings} or \emph{anti-Kirchhoff} vertex conditions. 
\end{exa}

\begin{defi}\label{CKKC}
For $\tilde{ \DD}$ defined as in~\eqref{deftilde} and with
\[
Y:={\rm Range}\,\tilde{ \DD}\qquad \hbox{or}\qquad Y:=\left({\rm Range}\,\tilde{ \DD}\right)^\perp,
\]
the boundary conditions 
\begin{equation*}
\left\{
\begin{array}{rcll}
\begin{pmatrix}u(0)\\u(1)\end{pmatrix}&\in& Y,\\
\begin{pmatrix}-u'(0)\\u'(1)\end{pmatrix}&\in& Y^\perp,\\
\end{array}
\right.
\end{equation*}
will be referred to as \textit{$CK$--condition} and \textit{$KC$--condition}, respectively. The second derivative with these conditions will be denoted by $\Delta^{CK}$ and $\Delta^{KC}$, respectively.
\end{defi}

Several features of the $KC$-condition -- including index theorems, trace formulae, convergence results, parabolic properties -- have been considered among others in~\cite{CheExn04,AlbCacFin07,FulKucWil07,Pos09,BolEnd09,CarMug09}. 
We will thoroughly discuss the eigenvalue multiplicities for these boundary conditions in Sections~\ref{adjcal} and~\ref{sec:antikir}.
 
\begin{rem}
If $H$ is finite-dimensional, say ${\rm dim}\; H=m$, and $R$ is positive semidefinite, then the $4m\times 2m$ algebraic system~\eqref{PYrankgen} has a nontrivial solution $(A,B)$ if and only if 
\begin{equation}\label{rankcond}
2m-{\rm rank}\;
\begin{pmatrix}
P_{Y^\perp}\begin{pmatrix}
{\Id}_{\mathbb C^m} & 0\\ \cos\sqrt{\lambda} {\Id}_{\mathbb C^m}& \sin\sqrt{\lambda}{\Id}_{\mathbb C^m}
\end{pmatrix}\\ 
P_{Y}\left(\sqrt{\lambda }\begin{pmatrix}
0 & {\Id}_{\mathbb C^m}\\ \sin\sqrt{\lambda}{\Id}_{\mathbb C^m} & -\cos\sqrt{\lambda}{\Id}_{\mathbb C^m}
\end{pmatrix}-R \begin{pmatrix}
{\Id}_{\mathbb C^m} & 0\\ \cos\sqrt{\lambda}{\Id}_{\mathbb C^m}& \sin\sqrt{\lambda}{\Id}_{\mathbb C^m}
\end{pmatrix}\right)
\end{pmatrix}
\ge 1.
\end{equation}
Hence, by Proposition~\ref{multipl0} each eigenvalue $\lambda$ can have multiplicity $m(\lambda)$ at most $2m$. More precise bounds on $m(\lambda)$ can be obtained in the special cases considered in Sections~\ref{adjcal} and~\ref{sec:antikir}.
\end{rem}
As the eigenvalues $\lambda_k$  interlace
with the Neumann and Dirichlet eigenvalues on $N$ uncoupled intervals $\alpha_k$ and $\omega_k$ respectively, i.e. $\alpha_k\leq \lambda_k \leq\omega_k$
for all $k\in\NN$, we can state the
\begin{cor}
\label{assymm} For the eigenvalues of the Laplacian  associated to the problem $\rm(EP_{Y,0})$  the following spectral asymptotic hold
$$\lim_{k\to\infty}\frac{\lambda_k}{k^2}=\frac{\pi^2}{N^2},
$$
where $\l_k$ denotes the $k$--th eigenvalue, and where the eigenvalues are counted according to their (geometric = algebraic) multiplicities.
\end{cor}

Let now $R=0$. We close this section focusing on the special case of $H={\mathbb C}$, i.e., of an elliptic problem concerning only one interval.
As long as we want to keep locality of the boundary conditions, there are only four possibilities:
\begin{itemize}
\item Dirichlet boundary conditions in both $0$ and $1$;
\item Neumann boundary conditions in both $0$ and $1$;
\item Dirichlet boundary condition in $0$ and Neumann boundary condition in $1$;
\item Neumann boundary condition in $0$ and Dirichlet boundary condition in $1$.
\end{itemize}
They correspond to the boundary conditions in $\rm(EP_Y)$ by means of four different subspaces $Y$ of $H\times H={\mathbb C}^2$: these are respectively
\begin{itemize}
\item $Y=\{0\}$;
\item $Y={\mathbb C}^2$;
\item $Y=$ subspace spanned by the vector $(0,1)$;
\item $Y=$ subspace spanned by the vector $(1,0)$.
\end{itemize}
If we generalize the above setting in order to allow for nonlocal interactions between the vertices at $0$ and $1$, we are actually performing an identification between two vertices: we might as well regard this setting as a differential operator on a loop around a vertex $v$. 

The boundary conditions in $\rm(EP_Y)$ are now given by a subspace of ${\mathbb C}^2$: besides the same four boundary conditions appearing in the interval case there are infinitely many new ones. All such boundary conditions interpolate between the Dirichlet--Dirichlet and the Neumann--Neumann case, i.e., $V_Y$ contains the form domain corresponding to the the Dirichlet case (i.e., $H^1_0(0,1)$) and is contained in the form domain corresponding to the Neumann case (i.e., $H^1(0,1)$). In particular, the operator associated with $(a_0,V_Y)$ is dominated by (resp., dominates) the second derivative with Dirichlet (resp., Neumann) boundary conditions, in the sense of self-adjoint operators. 

It seems that only a few publications are devoted to the study of differential operators on a loop. Among them, we mention~\cite{lus02}, where a characterization of a certain class of second order self-adjoint elliptic  operators is presented.

Already Proposition~\ref{multipl0} shows that the second derivative $\Delta_{Y,0}$ with associated eigenvalue problem $\rm(EP_{Y})$  is \emph{not} invertible if and only if $Y$ contains vectors of the form $(A,A)$ for some $A\in H$. In the case $H=\mathbb C$ this means that $\Delta_{Y,0}$ is  \emph{not} invertible in exactly two cases: $Y={\mathbb C}^2$ and the space $Y$ spanned by the vector $(1,1)$,
corresponding to (uncoupled) Neumann boundary conditions continuity/Kirchhoff conditions at the endpoints, respectively. In this section we are going to describe the spectrum of $\rm(EP_Y)$ more explicitly.

We focus on the boundary conditions defined by a 1-dimensional subspace $Y$ and thus neglect the trivial cases of $Y=\{0\}^2$ or $Y={\mathbb C}^2$ (standard Dirichlet or Neumann boundary conditions, whose associated spectrum is well-known). We can then consider conditions of the same form as in $\rm(EP_Y)$ with
$$Y\equiv Y_\alpha:=\left\langle \begin{pmatrix}
\alpha\\ 1
\end{pmatrix}\right\rangle\qquad \hbox{and}\qquad Y^\perp\equiv Y^\perp_\alpha=Y_{-\bar\alpha^{-1}}
$$
The associated orthogonal projections are given by
\begin{equation}\label{projformula}
P_Y=\kappa\begin{pmatrix}
\alpha & 1\\ 1 & \frac{1}{\alpha}
\end{pmatrix}\qquad \hbox{and}\qquad
P_{Y^\perp}=\kappa\begin{pmatrix}
 \frac{1}{\bar\alpha} & -1 \\ -1 & \bar\alpha
\end{pmatrix}
\end{equation}
for
$$\kappa:=\frac{\Real \alpha}{1+\Real \alpha}.$$ 


\begin{rem}
In order to consider the totality of possible boundary conditions it seems at first that we have to allow all $\alpha\in {\mathbb C}\cup \{\infty\}$, but in fact the eigenvalue problem $\rm(EP_{Y_\alpha})$, for some $|\alpha|>1$, can be equivalently realized (i.e., we have isospectrality) by applying the parity transformation 
$$u(1/2+x)\mapsto u(1/2-x),\qquad x\in [0,1],$$ 
and then switching to the eigenvalue problem $\rm(EP_{Y_{\alpha^{-1}}})$.
\end{rem}

\begin{prop}\label{main}
For all $\alpha\in \mathbb C$ the spectrum of $\rm(EP_{Y_\alpha})$
is given by
$$\left\{\lambda\ge 0 \setprop \cos\sqrt{\lambda}=\frac{2\Real \alpha}{1+|\alpha|^2}\right\}.$$
All the eigenvalues are simple for all $\alpha\not=\pm 1$. All the eigenvalues have multiplicity 2 for $\alpha=\pm 1$.
\end{prop}

\begin{proof}
First of all let us observe that, by Proposition~\ref{multipl0}, $0$ is an eigenvalue if and only if $\alpha=1$.

Now, plugging~\eqref{projformula} into~\eqref{rankcond} and eliminating the linearly dependent rows yields that $\lambda>0$ is an eigenvalue if and only if the matrix
\begin{equation}
\label{matrixalpha}
\begin{pmatrix}
1-\bar \alpha \cos \sqrt{\lambda} & -\bar \alpha \sin \sqrt{\lambda}\\
\sin \sqrt{\lambda} & \alpha-\cos \sqrt{\lambda}
\end{pmatrix}
\end{equation}
is singular, i.e., if and only if
$$2\Real \alpha = \cos\sqrt{\lambda} (1+|\alpha|^2).$$
Moreover, the eigenvalue $\lambda$ has multiplicity 2 if and only if the matrix in~\eqref{matrixalpha} vanishes: this is the case if and only if $\cos \sqrt{\lambda}=1$ and $\alpha=1$, or else $\cos \sqrt{\lambda}=-1$ and $\alpha=-1$.
\end{proof}


\begin{rem}
It follows from Proposition~\ref{main} that the spectrum associated with $Y_1$ (CK-condition, with the notation introduced in the Example~\ref{Ykirch-comp}) is given by
$\{(2k\pi)^2 \setprop k\in \mathbb N_0\}$ for $\alpha= 1,$
while the spectrum associated with $Y_{-1}$ (KC-condition) agrees with
$\{((2k+1)\pi)^2 \setprop k\in \mathbb N_0\}$ for $\alpha= -1$. 
This is of course nothing but a special case of Corollary~\ref{multipl0cor}.
Hence, among all second derivatives with boundary conditions as in $\rm(EP_{Y_\alpha})$, those corresponding to $\alpha=1$ (i.e., to the $CK$-condition) have the smallest possible lowest eigenvalue; and therefore, those corresponding to $\alpha=-1$ (i.e., to the $KC$-condition) have the largest possible lowest eigenvalue. Of course, these considerations have a counterpart in the asymptotics of the associated semigroups that govern the heat equation corresponding with $\rm(EP_{Y_\alpha})$.
\end{rem}

\section{Adjacency calculus and spectral analysis under the $CK$--condition}
\label{adjcal} 
In this section we recall some known results \cite{Bel85,PNE} on the Laplacian on a graph by means of combinatorial
quantities when continuity and Kirchhoff type conditions are imposed in the vertices.
For that purpose, let us introduce some terminology.
For any graph $\Gamma =(V,E,\in)$, the vertex set is denoted by
$V=V(\Gamma)$, the edge set by $E=E(\Gamma)$ and the incidence
relation by $\in \subset V\times E$. The \emph{degree} of a vertex
$v$ is defined by $\gamma(v)= \cd \{e\in E\setprop v\in e\}$.
All graphs considered in this and the following section are
assumed to be nonempty and finite with
$$ n=\# V,\qquad N=\# E$$
and $c=c(\Gamma)$ connected components $\Gamma_1,\ldots,\Gamma_c$. 
\begin{defi}
\begin{align*}
	c^+(\Gamma) &= \# \left\{\Gamma_k\setprop \Gamma_k \text{ bipartite}\right\}\\
	c^-(\Gamma) &= \# \left\{\Gamma_k\setprop \Gamma_k \text{ non-bipartite}\right\}
\end{align*}
\end{defi}

The graphs $\Gamma$ are also assumed to be \emph{simple}, i.e.\ $\Gamma$ contains no loops, and
at most one edge can join two vertices in $\Gamma$. By definition,
a \emph{circuit} is a connected and regular graph of degree $2$. We
number the vertices by $v_1,\ldots, v_n$, the respective degrees
by $\gamma_1,\ldots, \gamma_n$, and the edges by $e_1,\ldots,
e_N$. The \emph{adjacency matrix}
$\AD(\Gamma)=\left(e_{ih}\right)_{n\times n}$ of the graph is
defined by
$$e_{ih}=
\begin{cases}
1  & \text{if } v_i \text{ and } v_h \text{ are adjacent in }\Gamma\\
                 0  & \text{else}\\
\end{cases}
$$
Note that $\AD(\Gamma)$ is indecomposable if and only if $\Gamma$ is
connected. By simplicity, any two adjacent vertices $v_i$ and
$v_h$ determine uniquely the edge $e_s$ joining them, and we can
set
$$s(i,h)=
\begin{cases}
s  & \text{if } e_s\cap V =\{v_i,v_h\},\\
  1  & \text{otherwise.}
\end{cases}
$$
For further graph theoretical terminology we refer to \cite{wi}, and for the algebraic graph theory to~\cite{big,cds}.

The abstract graph $\Gamma$ can be realized as a  \emph{topological graph} in $\RR^m$ with $m\geq 3$, i.e. $V(\G)\subset
\RR^m$ and the edge set consists in a collection of Jordan curves
$E(\G)=\{ \pi_j:[0,1]\to \RR\setprop 1\leq j \leq N\}$.
The trace
of the topological graph leads to the associated \textit{network} or \textit{metric graph}
$G=\bigcup_{j=1}^N\pi_j\left([0,1]\right)$. The arc
length parameter of the edge $e_j$ is denoted by $x_j$. Clearly, for many purposes
the $\pi_j$ have to display certain regularity properties, say $\pi_j\in\CCC^2$, but
for the present context it suffices to identify each edge $e_j$ with $[0,1]$. 
We shall distinguish the \emph{boundary vertices}
$V_b=\{v_i\in V\setprop \gamma_i=1\}$ from the \emph{ramification
vertices}
$V_r=\{v_i\in V\setprop\gamma_i\geq 2\}$.
The \textit{orientation} of the graph $\Gamma$ is encoded in the \emph{signed incidence matrix},
 which in this setting can be written as  $ \DD(\Gamma)=\left(d_{ij}\right)_{n\times N}$ with
\begin{equation*}
\label{orifa} d_{ij}=
\begin{cases}
1  & \text{if } \pi_j(1)=v_i,\\
-1  & \text{if } \pi_j(0)=v_i,\\
0  & \text{otherwise.} \\
\end{cases}
\end{equation*}
In fact, the factors $d_{ij}$ stand for the outer normal derivative
at $v_i$ with respect to $e_{j}$.
Functions on the graph or on the network are mappings $u=\left(u_j\right)_{N\times
1}:[0,1] \to \CC^N$ with edge components
$u_j:[0,1]\to\CC$. Here we and use the abbreviations
$$u_j(v_i):=u_j(\pi_j^{-1}(v_i)),\quad
\partial_j u_j(v_i):=\frac{\partial}{\partial x_j}u_j(x_j)\Bigr|_{\pi_j^{-1}(v_i)}
\quad\hbox{etc.}
$$
It is well--known that $\text{\rm corank}(\Gamma):=\dim_\CC\ker\DD(\Gamma)=N-n+c(\Gamma)$,
and that $\text{\rm corank}(\Gamma)=1$ if and only if $\Gamma$ is \textit{unicyclic}, i.e.
by definition,  $\Gamma$ contains exactly one circuit. In the connected case, $\Gamma$ is
unicyclic if and only if $N=n$.
 
In this short section we consider the standard Laplacian
$$\Delta
=\left(\left(u_j\right)_{N\times
1}\mapsto \left(\partial_j^2 u_j\right)_{N\times
1}\,\right):\CCC^2[0,1]^N\to\CCC[0,1]^N
$$
under the \emph{continuity condition}
\begin{equation}
\label{1004} \forall v_i\in V_r\,:\,\,e_j\cap e_s
=\{v_i\}\;\Longrightarrow \;u_j(v_i)=u_s(v_i),
\end{equation}
and under  the \emph{Kirchhoff flow condition}
\begin{equation}
\label{gkc} \sum_{j=1}^N d_{ij} 
\partial_j u_j
(v_i)=0  \quad \text{for}\quad 1\leq i\leq n.
\end{equation}
As anticipated in Section~\ref{eigen1}, we shall refer to Conditions~\eqref{1004} and~\eqref{gkc} as to the \textit{$CK$--condition}.
It can be readily written as a canonical boundary condition in the sense of H\"older as follows, see \cite{be2}.
Define the matrix
$$\mathbf{S_2}=
\begin{pmatrix}P_{00} &P_{01}\\
P_{10} &P_{11}\end{pmatrix}
_{2N\times 2N}$$
with matrices $P_{\a\b}=(p_{\a\b jk})_{N\times N}$ defined by
$$p_{\a\b jk}=\begin{cases}
1 &\text{if } \pi_j(\a l_j)=\pi_k(\b l_k),\\
0 &\text{otherwise}.\end{cases}$$
The matrix $\mathbf{S}_2$ is symmetric and decomposable into a block diagonal 
matrix with $n$ blocks, each of them being a dyad of the form 
$\be_{\g_i}\be_{\g_i}^*$, where throughout $\be_k=(1)_{k\times 1}$. Let $\mathbf{S}_1$ denote the orthogonal 
projection onto the kernel of $\mathbf{S}_2$. In detail, for $d_{ij}=-1$ or 
$d_{i,j-N}=+1$ with $1\leq j\leq 2N$ and $z=(z_j)_{2N\times 1}$ we have
$$(\mathbf{S}_1 z)_j=z_j-\,\frac{1}{\g_i} \sum \{z_k\setprop 1\leq k\leq 
2N,d_{ik}=-1\;\hbox{or}\; d_{i,k-N}=+1\}\; .$$
Now the continuity condition at ramification vertices reads $\mathbf{S}_1 \begin{pmatrix}u(0)\\u(1)\end{pmatrix}=0$, and
the $CK$--Condition takes the form
$$\begin{pmatrix}u(0)\\u(1)\end{pmatrix}\in Y:= \ker\mathbf{S}_1,\qquad
\begin{pmatrix}-u'(0)\\u'(1)\end{pmatrix}\in Y^\perp=\ker\mathbf{S}_2.
$$
This shows that we are in fact inside the general setting presented in Section~\ref{eigen1}.
It is well known that the eigenvalues of
$-\Delta$ are nonnegative, see e.g.\ Proposition~\ref{multipl0} or \cite{Bel85}.
Following the transformations in
\cite{Bel85,PNE} the eigenvalue problem for $\Delta$ is equivalent
to the matrix differential boundary eigenvalue problem
\eqref{mevp1}--\eqref{mevp6} below incorporating the adjacency
structure of the network. For that purpose we recall that the
Hadamard product of matrices of the same size is defined as
$\left(a_{ik}\right) \star \left(b_{ik}\right)=\left(a_{ik}
b_{ik}\right)$. 
For a function $u:G \to \RR$ denote its value
distribution in the vertices by
\begin{equation}
\label{nodi} \f={\bf n}(u)=\left(u(v_i)\right)_{n\times 1},
\end{equation}
and for $x\in[0,1] $ define the matrix
$$
\uu(x)=\left(u_{ih}(x)\right)_{n\times n}$$ by
\begin{equation}
\label{tra} u_{ih}(x)=e_{ih}u_{s(i,h)}\left(\frac{1+d_{is(i,h)}}{2}-xd_{is(i,h)}\right).
\end{equation}
Then the eigenvalue problem for $\Delta$ under~\eqref{1004} and~\eqref{gkc} reads as follows.\\
\begin{eqnarray}
u_{ih}\in C^2([0,1])& &\text{ for all } 1\leq i,h\leq n\label{mevp1}\\
 e_{ih}=0\Rightarrow u_{ih}=0& &\text{ for all } 1\leq i,h\leq n\label{mevp2}\\
\uu''=-\lambda \uu & &\text{ in } [0,1]\label{mevp3}\\
 \uu(0)= \f \,\be^* \star \AD & &\text{ continuity in }
V_r(\Gamma)\label{mevp4}\\
 \uu^t(x)=\uu(1-x) & &\text{ for }
x\in[0,1]\label{mevp5}\\
 \uu'(0)\be=0 & &\text{ Kirchhoff flow condition }\label{mevp6}
\end{eqnarray}
\\
\noindent Set
$$\Phi:=\uu(0)=\f \,\be^* \star \AD,\qquad \Psi:=\uu'(0).$$
Finally, introduce the row--stochastic \emph{transition} matrix
\begin{equation}\label{zz}
\ZD:=\hbox{\rm
Diag}\left(\AD\,\be\right)^{-1}(\AD),
\end{equation}
that has only real eigenvalues \cite{Bel85} and plays a key role in the characterization  of the spectrum of
the Laplacian. As for the multiplicities recall the
following, which has been obtained in~\cite{Bel85,PNE}.

\begin{theo}\label{mgeoalg} If $\lambda$ is an eigenvalue of $-\Delta$ in
$\CCC_K^2(G)$ and $\f\in\RR^n$ is a vertex distribution of an
eigenfunction belonging to $\lambda$, then
\begin{equation}\label{3100ck}
\ZD \f = \cos{\sqrt\lambda}\,\f.
\end{equation}
Conversely, if $\cos \sqrt\lambda$ is an eigenvalue of $\ZD$
admitting the eigenvector $\f\in\RR^n$, then $\lambda$ is an
eigenvalue of $-\Delta$ in $\CCC_K^2(G)$ and $\f$ the vertex
distribution of some eigenfunction belonging to $\lambda$. The
multiplicities are
\begin{equation*}
m_g(\lambda)=m_a(\lambda)=
\begin{cases}
c(\Gamma)&\text{ if } \lambda= 0,\\
m_g(\cos\sqrt\lambda,{\ZD})=m_a(\cos\sqrt\lambda,{\ZD})&\text{ if } \sin{\sqrt\lambda}\ne 0,\\
\text{\rm corank}(\Gamma) + c(\Gamma)=N-n+2c(\Gamma)
&\text{ if } \cos{\sqrt\lambda}=1,\,\lambda>0, \\
\text{\rm corank}(\Gamma) + c^+(\Gamma)-c^-(\Gamma)=N-n+2c^+(\Gamma)
&\text{ if } \cos{\sqrt\lambda}=-1,\,\lambda>0.
\end{cases}
\end{equation*}
\end{theo}

\begin{rem}
In other words, in the relevant case of a connected graph the spectrum of $-\Delta$ under the $CK$-condition can be partitioned
as
$$\sigma(-\Delta)=\{0\} \cup \sigma_{i}\cup\sigma_s,$$
where
\begin{equation*}
\sigma_i :=\left\{\left( 2\ell\pi \pm \mathrm{arc}\cos \alpha
\right)^{2}\setprop \alpha\in\sigma(\mathcal{Z})\setminus \{-1,1\} \text{
and }\ell\in\mathbb{Z} \right\} ,
\end{equation*}
and
\begin{equation*}
\sigma_s:= \{k^2\pi^2\setprop k\in\mathbb{Z}\setminus \{0\}\}.
\end{equation*}
As the elements of $\sigma_i$ are solely determined by the adjacency structure of the graph, they are sometimes called ``immanent eigenvalues''. For the multiplicities of the eigenvalues we have for all $k\in \mathbb Z\setminus\{0\}$
\begin{enumerate}
\item $m(0)=1$; \item $m(\lambda)=\dim \ker \left( \cos\sqrt{\lambda }+\mathcal{Z}\right) $ for $\lambda
\in \sigma_i$; \item $m(k^2\pi^2)=N-n+2$ if $\Gamma$  is bipartite and
$m((2k+1)^2\pi^2)=N-n$ if $\Gamma$ is not bipartite; \item $m(4k^2\pi^2)=N-n+2$.
\end{enumerate}
\end{rem}

A spectral asymptotics of Weyl type follows promptly.

\begin{cor}(\cite{Bel85,PNE})
\label{aslapcan} For the canonical Laplacian $-\Delta$ under
the $CK$--condition the spectral asymptotics
$$\lim_{k\to\infty}\frac{\lambda_k}{k^2}=\frac{\pi^2}{N^2}
$$
holds, where $\l_k$ denotes the $k$--th eigenvalue, and where the eigenvalues are counted according to their (geometric = algebraic) multiplicities.
\end{cor}

By Theorem~\ref{recover} below, a necessary conditions for the \emph{whole} spectra of $-\Delta^{CK}$ for two different graphs $\Gamma_1$ and $\Gamma_2$ to be comparable is that both graphs have the same number of vertices, edges, bipartite and non-bipartite components. This is of course a very strong assumption. Thus, a more natural question is whether some relevant spectral subset -- i.e., the lowest non-trivial eigenvalue -- of $-\Delta^{CK}$ is lowered or raised by certain graph operations. For instance, the following comparison result for eigenvalues of $\Delta^{CK}$ is a direct consequence of Theorems~\ref{mgeoalg} and~\ref{mak},~\eqref{transnormlapl} and a result by Chung~\cite[Lemma~1.15]{Chu97}.

\begin{prop}\label{compar-ck}
Let $\Gamma_1,\Gamma_2$ be two connected graphs such that $\Gamma_2$ be  formed by contractions of vertices from $\Gamma_1$. Then, the lowest nontrivial eigenvalue of $-\Delta^{CK}$ on the metric graph $G_1$ associated with ${\Gamma_1}$ is at most as large as  the lowest nontrivial eigenvalue on the metric graph $G_2$ associated with ${\Gamma_2}$.
Hence, the heat equation under $CK$-conditions on $G_1$ converges to the equilibrium faster than on $G_2$.
\end{prop}

\section{Spectral analysis of the $KC$--condition}\label{sec:antikir}
In this section, we consider the standard Laplacian $$\Delta
=\left(\left(u_j\right)_{N\times
1}\mapsto \left(\partial_j^2 u_j\right)_{N\times
1}\,\right):\CCC^2[0,1]^N\to\CCC[0,1]^N
$$
under the so-called \textit{anti-Kirchhoff condition}, i.e.
\begin{equation}
\label{ak1} \sum_{j=1}^N d_{ij}^2 \ u_j(v_i)=0  \quad \text{for}\quad 1\leq i\leq n
\end{equation}
and
\begin{equation}
\label{ak2} \forall v_i\in V_r\,:\,\,e_j\cap e_s
=\{v_i\}\;\Longrightarrow \;d_{ij}u_j(v_i)=d_{is}u_s(v_i),
\end{equation}
As in Definition~\ref{CKKC} we refer to~\eqref{ak1} and~\eqref{ak2} as the \textit{$KC$--condition}.
In terms of the conditions discussed in Section~\ref{eigen1}, they correspond to
$$\begin{pmatrix}u(0)\\u(1)\end{pmatrix}\in Y:= \ker\mathbf{S}_2,\qquad
\begin{pmatrix}-u'(0)\\u'(1)\end{pmatrix}\in Y^\perp=\ker\mathbf{S}_1,
$$
where $\mathbf{S}_1,\mathbf{S}_2$ are defined as in Section~\ref{adjcal}. 
It follows from Proposition \ref{multipl0} that the eigenvalues of $-\Delta$ under these transition conditions
are nonnegative.
Using the adjacency transforms from Section \ref{adjcal}, the
eigenvalue problem for $\Delta$ reads as follows.\\
\begin{eqnarray}
u_{ih}\in C^2([0,1])& &\text{ for all } 1\leq i,h\leq n\label{akevp1}\\
 e_{ih}=0\Rightarrow u_{ih}=0& &\text{ for all } 1\leq i,h\leq n\label{akevp2}\\
\uu''=-\lambda \uu & &\text{ in } [0,1]\label{akevp3}\\
\uu^t(x)=\uu(1-x) & &\text{ for }
x\in[0,1]\label{akmevp5}\\
\uu(0)\be=0 & & \text{~\eqref{ak1} }\label{akevp6}\\
 \uu'(0)= \psi \,\be^* \star \AD & &\text{~\eqref{ak2} }\label{akevp4}
\end{eqnarray}
\\
\noindent As above, set
$\Phi:=\uu(0)$, $\Psi:=\uu'(0)=\psi \,\be^* \star \AD$ and 
$\ZD=\hbox{\rm
Diag}\left(\AD\,\be\right)^{-1}\AD$. Then we can state the following.

\begin{theo}\label{mak}
1) If $\lambda>0$ is an eigenvalue of $-\Delta$ under the $KC$--condition in
$\CCC^2[0,1]^N$ and $\psi\in\RR^n$ is a vertex distribution of the normal derivatives of an
eigenfunction belonging to $\lambda$, then
\begin{equation}\label{3100}
\ZD \psi = -\cos{\sqrt\lambda}\,\psi.
\end{equation}
Conversely, if $\lambda>0$ and $-\cos \sqrt\lambda$ is an eigenvalue of $\ZD$
admitting the eigenvector $\psi\in\RR^n$, then $\lambda$ is an
eigenvalue of $-\Delta$ under the $KC$--condition in $\CCC^2[0,1]^N$ and $\psi$ the
distribution of normal derivatives of some eigenfunction belonging to $\lambda$. \\
2) The multiplicities of the eigenvalues are given by
\begin{equation*}
m_g(\lambda)=
\begin{cases}
\text{\rm corank}(\Gamma) -c^-(\Gamma)=N-n+c^+(\G)&\text{ if } \lambda =0,\\
m_g(-\cos\sqrt\lambda,{\ZD})=m_a(-\cos\sqrt\lambda,{\ZD})&\text{ if } \sin{\sqrt\lambda}\ne 0,\\
\text{\rm corank}(\Gamma) + c(\Gamma)=N-n+2c(\G)
&\text{ if } \cos{\sqrt\lambda}=-1,\,\lambda>0,\\
\text{\rm corank}(\Gamma) + c^+(\Gamma) - c^-(\Gamma)=N-n +2c^+(\G)
&\text{ if } \cos{\sqrt\lambda}=1,\,\lambda>0.
\end{cases}
\end{equation*}
\end{theo}

Let us recall  from \cite{Bel85} that for $\Gamma$ connected, i.e. $c(\G)=1$,
\begin{align}
	&\dim_\RR\MM^-(\G)=\text{corank}(\Gamma)=N-n+1,\label{6001}\\
	&\dim_\RR\MM^+(\G) = \begin{cases}
N-n+1
&\text{ if }  \Gamma \text{ bipartite,}\\
N-n&\text{ if } \, \Gamma \text{ is not
bipartite,}\label{6002}
\end{cases}
\end{align}
where
\begin{align*}
&\MM(\G):=\left\{M\setprop M=\left(m_{ih}\right)_{n\times n},\
\forall i,h\in\left\{1,\ldots,n\right\} :\,
 e_{ih}=0\Rightarrow m_{ih}=0\right\},\\
	&\MM^-(\G):=\left\{M\in \MM(\G)\setprop M^*=-M,\,M\be=0\right\},\\
	&\MM^+(\G):=\left\{M\in \MM(\G)\setprop M^*=M,\,M\be=0\right\}.
\end{align*}

\begin{proof} [Proof of Theorem~\ref{mak}] As $\text{\rm corank}(\Gamma) = \sum_{k=1}^c \text{\rm corank}(\Gamma_k) $, we can confine ourselves to the case $c(\Gamma)=1$. First, let us consider the case
$$\l=0.$$
Then each eigenfunction of~\eqref{akevp1}--\eqref{akevp4} is of the form
\begin{equation}\label{adjantik}
\uu(x)=\Phi+x\Psi\quad\text{ with }\Phi,\Psi\in\MM(\G). 
\end{equation}
Reasoning as in the proof of Proposition~\ref{multipl0}, the slope matrix $\Psi$ has to vanish, 
since $u$ has to be constant on each edge:
$$
\begin{array}{ll}
0=\displaystyle\sum_{j}
\int^{1}_0  \left(\partial_j^2 u_j\right)u_j dt_j &
=\displaystyle -\sum_{j}\int^{1}_0 \left(\partial_j u_j\right)^2 dt_j+
\displaystyle\sum_{j}
\left[u_j\partial_j u_j\right]_0^{1}\\
& = \displaystyle -\sum_{j}\int^{1}_0 \left(\partial_j u_j\right)^2 dt_j
 +\sum_{i} \bigl( \underbrace{d_{ij} \partial_j u_j}_{\text{independent of } j}
\underbrace{\displaystyle \sum_{j}  u_j (v_i)}_{=0}\bigr).\\
\end{array}
$$
Clearly, $\Phi\in\MM^+(\G)$, and any matrix belonging to $\MM^+(\G)$ defines an eigenfunction to $\l=0$. Thus, the eigenspace $E_0(\Gamma,\Delta)$ is isomorphic to $\MM^+(\G)$ which shows the multiplicity formula for $\l=0$.

Now suppose
$\l>0.$
In this case, a fundamental solution of~\eqref{akevp3} is given by
\begin{equation}
\label{fuso}
\uu(x)=\cos(x\sqrt \l)\Phi+ \frac{\sin(x\sqrt \l)}{\sqrt \l}\Psi.
\end{equation}
Thus, 
$$\uu'(x)=-\sqrt{\l}\sin(x\sqrt{\l})\Phi+ \cos(x\sqrt{\l})\Psi,\quad \uu'(1)=-\sqrt{\l}\sin(\sqrt{\l})\Phi+ \cos(\sqrt{\l})\Psi=-\Psi^*$$
and, by~\eqref{akmevp5},
$$\uu(1)=\Phi^*=\Phi\cos\sqrt \l + \frac{\sin\sqrt \l}{\sqrt \l}\Psi. $$
In the case 
$$\sin\sqrt \l\ne 0,$$
we can conclude that
$$\Phi= \frac{1}{\sqrt \l\sin\sqrt \l}
\left(\cos\sqrt \l \ \Psi +  \Psi^* \right)\star\AD,
$$
and, using $\Phi\be=0$ and  $\Psi= \psi \,\be^* \star \AD $, we are led to
$$\left(\AD\star \be \psi^*\right)\be + \cos\sqrt \l\,
\left(\AD\star\psi \be^* \right)\be=0,
$$
or the \textit{characteristic equation}
\begin{equation}
\label{chareqa} \ZD(\Gamma)\psi=-\cos \sqrt\l \ \psi.
\end{equation}
This shows the multiplicity formula for $\sin\sqrt \l\ne 0$. For the remaining case, suppose first
$$\cos\sqrt \l= -1.$$
Then $\Phi^*=-\Phi$, and~\eqref{akevp6} and~\eqref{6001} imply that $m_g(\lambda)\geq N-n+1$,
using eigenfunctions with vanishing matrix $\Psi$.
All eigensolutions vanishing in all vertices are of the form $\frac{\sin(x\sqrt \l)}{\sqrt \l}\Psi$
with a matrix $\Psi=\Psi^*=\psi \,\be^* \star \AD$ that has to be a multiple of $\AD$. This shows
$m_g(\lambda)=N-n+2$. If instead
$$\cos\sqrt \l= 1,$$
then $-\Psi^*=\Psi=\psi \,\be^* \star \AD$ can only be non-vanishing, and eigensolutions of the form $\frac{\sin(x\sqrt \l)}{\sqrt \l}\Psi$
can only exist if $\Gamma$ is bipartite. On the other hand, $\Phi^*=\Phi$, and~\eqref{akevp6} and~\eqref{6002} imply that
the eigenspace $E_{\lambda}(\Gamma,\Delta)$ contains a subspace isomorphic to $\MM^+(\G)$
using eigenfunctions with vanishing matrix $\Psi$. Thus, by~\eqref{6002}, $m_g(\lambda)=N-n+2$
in the bipartite case and $m_g(\lambda)=N-n$ if $\Gamma$ is not bipartite.\end{proof}

\begin{rem}
1) In the first part of Theorem~\ref{mak} we cannot drop the assertion that $\lambda>0$, whether the graph is bipartite or not. This can be seen considering the following graphs
\begin{center}
\begin{tikzpicture}[style=thick]
\draw (-5,0.7) -- (-6.4,0.7);
\draw[fill] (-5,0.7) circle (2pt);
\draw[fill] (-6.4,0.7) circle (2pt);
\draw (0,0) -- (0,1.4) -- (1,0.7) -- cycle;
\draw (0,0) -- (-1,0.7) -- (0,1.4);
\draw[fill] (0,0) circle (2pt);
\draw[fill] (0,1.4) circle (2pt);
\draw[fill] (1,0.7) circle (2pt);
\draw[fill] (-1,0.7) circle (2pt);
\end{tikzpicture}
\end{center}
In the first case, $-1$ is an eigenvalue of $\mathcal Z$ but $0$ is not an eigenvalue of $\Delta^{KC}$. In the second one, $-1$ is not an eigenvalue of $\mathcal Z$ but the function of constant value $(1,-1,0,-1,1)$ is an eigenfunction of $\Delta^{KC}$ for the eigenvalue $0$.

2) Like in the case of $CK$-condition, the above result shows that in the relevant case of a connected graph the spectrum of $\Delta^{KC}$ can be partitioned as
$$\sigma(-\Delta)=\{0\} \cup \sigma_{i}\cup\sigma_s,$$
where
\begin{equation*}
\sigma_i :=\left\{\left( 2\ell\pi \pm \mathrm{arc}\cos \alpha
\right)^{2}\setprop -\alpha\in\sigma(\mathcal{Z})\setminus \{-1,1\} \text{
and }\ell\in\mathbb{Z} \right\} ,
\end{equation*}
is the set of the immanent eigenvalues and the elements of
\begin{equation*}
\sigma_s:= \{k^2\pi^2\setprop k\in\mathbb{Z}\setminus \{0\}\}
\end{equation*}
correspond to the spectrum of the single edge problem.
For the multiplicities of the eigenvalues we have
\begin{enumerate}
\item $m(0)=N-n+1$ if $\Gamma$  is bipartite and $m(0)=N-n$ if $\Gamma$ is not bipartite; \item $m(\lambda)=\dim \ker \left( \cos\sqrt{\lambda }-\mathcal{Z}\right) $ for $\lambda
\in \sigma_i$; \item $m(k^2\pi^2)=N-n+2$ if $\Gamma$  is bipartite and
$m((2k+1)^2\pi^2)=N-n$ if $\Gamma$ is not bipartite; \item $m(4k^2\pi^2)=N-n+2$.
\end{enumerate}
\end{rem}

\begin{rem}\label{rem:tutte} 1) Note that on trees there are no harmonic functions under the $KC$--condition.

2) Bipartite graphs are exactly those graphs that can be endowed with an orientation such that each vertex is either a sink (only incoming
incident edges) or a source (only outgoing incident edges).
In this way harmonic functions with respect to the $KC$--condition are edgewise constant functions that satisfy the Kirchhoff rule at each node, i.e.
 ${\mathbb C}$--flows in the sense of~\cite[\S IX.4]{Tut84}. In fact, these edge distributions are exactly the elements of the null space of the incidence matrix $\DD(\Gamma)$. Now, it is well--known that the nullity of this matrix is $N-n+c(\Gamma)$, cf. e.g.~\cite[\S I.5]{big}.
This yields another proof for the claimed value of $m_g(0)$.

3) In the non-bipartite case, the above characterization fails to hold since the Kirchhoff condition for the potentials on the edges does not take into account
the orientation, while it is essential in the definition of the circuit space. To overcome this problem, observe that harmonic functions with respect to the KC-condition are edgewise constant functions that satisfy the Kirchhoff rule at each node (regardless of the bipartiteness of the graph). In fact, their edge distributions are exactly the elements of the null space of the \emph{unsigned} incidence matrix of $\Gamma$. Now, it is well--known that the nullity of this matrix is $N - n + c^+(\Gamma)$, cf. [32]. This yields another proof for the claimed value of $m_g(0)$.
\end{rem}

In the same way as for the continuity~\eqref{1004} under general consistent Kirchhoff flow conditions \cite{PNE}, the above multiplicity formulae determine the asymptotic behavior of the eigenvalues of the elliptic problem $(EP_{Y,R})$.
\begin{equation}
\label{KCR}
\begin{pmatrix}u(0)\\u(1)\end{pmatrix}\in Y= \ker\mathbf{S}_2,\qquad
\begin{pmatrix}-u'(0)\\u'(1)\end{pmatrix}+R\begin{pmatrix}u(0)\\u(1)\end{pmatrix}\in Y^\perp=\ker\mathbf{S}_1,
\end{equation}
where $R$ is a $2N\times2N$ hermitian matrix.
\begin{cor}
\label{aslapcanak} For the canonical Laplacian $-\Delta$ under the $KC$-condition the spectral asymptotics
$$\lim_{k\to\infty}\frac{\lambda_k}{k^2}=\frac{\pi^2}{N^2}
$$
holds, where $\l_k$ denotes the $k$--th eigenvalue, and where the eigenvalues are counted according to their (geometric = algebraic) multiplicities.
\end{cor}
For a general, not necessarily hermitian matrix $R$, one can follow the Liouville transform approach in \cite{bl5} counting
the algebraic multiplicities in order to get the same asymptotic formula. We omit the details here.
It has been shown in \cite{Bel85} that $\Gamma$ is bipartite if and only if $-1$ is an eigenvalue
of the matrix $\ZD$ or, equivalently, the spectrum of $\ZD$ is symmetric with respect to $0$
counting multiplicities. This can be used in various applications.

\begin{cor}
\label{ckequalkc} Suppose that the graph $\Gamma$ is connected. Then the spectra of $-\Delta^{CK}$ and $-\Delta^{KC}$ coincide, counting
multiplicities, if and only if $\Gamma$ is unicyclic and bipartite. 
\end{cor}

\begin{proof}
If $\Gamma$ is unicyclic and bipartite, then $\text{corank}(\G)=1$ and the assertion follows by
Theorems \ref{mgeoalg} and \ref{mak}. For the reverse implication, by hypothesis, the spectrum of $\ZD$ is symmetric with respect to $0$, which implies that $\Gamma$ is bipartite. 
Moreover, since $\text{corank}(\G)=1$, i.e. $N=n$, and $\G$ contains exactly one circuit. 
\end{proof} 

In the disconnected case, both spectra  coincide counting
multiplicities, if and only if $\Gamma$ is bipartite and each connected component is unicyclic,
since $N=n$ if and only if $\text{corank}(\Gamma)=c(\Gamma)$.

\begin{cor}
\label{akgleichk} The graph $\Gamma$ is bipartite if and only if the network immanent eigenvalues ($\sin\sqrt \l\ne 0$) of
$-\Delta^{CK}$ and $-\Delta^{KC}$ are the same, counted according to their multiplicities. 
\end{cor}

\begin{rem}

If $\Gamma$ has a component with at least two edges, then $\mathcal Z$ induces an immanent eigenvalue between $0$ and the first non-trivial, non-immanent eigenvalue $\pi^2$. Since in the connected case the heat semigroup under the $CK$-condition always converges towards an equilibrium, while the semigroup under the $KC$-condition only fails to do so if $\G$ is unicyclic with a circuit of odd length, this shows that the rate of convergence of the semigroup is usually determined by connectivity of the graph. Moreover, this eigenvalue coincides with $\mu:=(\arccos \alpha)^2$ and $\nu:=(\arccos \beta)^2$  (for the $CK$-condition and the $KC$-condition, respectively), where $\alpha$ and $\beta$ are the largest and the smallest eigenvalues of $\mathcal Z$ different from $\pm 1$, respectively. Observe that in the connected case, the eigenvalue $1$ is necessarily simple by the Perron--Frobenius theorem. 
Hence, the eigenvalues of the dual pair $\rm(EP_{CK}),(EP_{KC})$ come in pairs and, in the particular case of bipartite graphs, even in quadruples (because then the spectrum of $\mathcal Z$ is symmetric with respect to the origin).

\begin{figure}[h]\label{plotd}
\begin{center}

\begin{tikzpicture}[domain=0:10]
    \draw[->] (-0.2,0) -- (11.2,0) node[right] {$\lambda$};
    \draw[->] (0,-2.4) -- (0,2.4);
    \draw[color=red] plot[id=-cossqrt] function{-2*cos(sqrt(x*pi))} 
        node[right] {$-\cos\sqrt{\lambda}$};
    \draw[color=blue] plot[id=cossqrt] function{2*cos(sqrt(x*pi))} 
        node[right] {$\cos\sqrt{\lambda}$};

\draw (0,2) -- (0,2) node[left=3pt] {$1$};      
 \draw[dotted][color=gray](0,2) -- (10.2,2);

\draw (0,1.414) -- (0,1.414) node[left=3pt] {$\alpha$};
\draw (0,1.2) -- (0,1.2) node[left=3pt] {$\vdots$};
    \draw[dotted][color=gray] (0, 1.414) -- (10.2, 1.414);

\draw (0,-.5) -- (0,-.5) node[left=3pt] {$\vdots$};
\draw (0,-1) -- (0,-1) node[left=3pt] {$\beta$};
    \draw[dotted][color=gray] (0, -1) -- (10.2, -1);
    
\draw (pi ,-.1) -- (pi ,.1) node[below=3pt] {\tiny $\pi^2$};
 \draw[dotted][color=gray] (pi,0) -- (pi,2);
\draw (0,-.1) -- (0,.1) node[left=3pt] { $0$};

\draw[fill][color=black] (0, 2) circle (1.7pt);
\draw[fill][color=blue] (0, 2) circle (.8pt);
\draw[color=blue] (0,0) circle (2.5pt);

\draw[fill][color=black] (0, 1.414) circle (1.5pt);
\draw[fill][color=blue] (pi/16, 1.414) circle (.9pt);
\draw[color=blue] (pi/16, 0) circle (2.5pt);
\draw (pi/16,-.1) -- (pi/16,.1) node[above] {\tiny $\mu$};
 \draw[dotted][color=gray] (pi/16,0) -- (pi/16,1.414);
 
\draw[fill][color=black] (0, -1) circle (1.5pt);
\draw[fill][color=red] (pi/9, -1) circle (.9pt);
\draw[color=red] (pi/9, 0) circle (1.5pt);
\draw (pi/9,-.1) -- (pi/9,.1) node[below=4pt] {\tiny $\nu$};
  \draw[dotted][color=gray] (pi/9,-1) -- (pi/9,0);
  
\draw[fill][color=red] (9*pi/16, 1.414) circle (.9pt);
\draw[color=red] (9*pi/16, 0) circle (1.5pt);
\draw (9*pi/16,-.1) -- (9*pi/16,.1) node[below=1pt] {\tiny $(\pi-\sqrt{\mu})^2$};
  \draw[dotted][color=gray] (9*pi/16,0) -- (9*pi/16,1.414);
  
\draw[fill][color=blue] (4*pi/9, -1) circle (.9pt);
\draw[color=blue] (4*pi/9, 0) circle (2.5pt);
\draw (4*pi/9,-.1) -- (4*pi/9,.1) node[above=1pt] {\tiny $(\pi-\sqrt{\nu})^2$};
  \draw[dotted][color=gray] (4*pi/9,-1) -- (4*pi/9,0);
  
\draw[fill][color=red] (25*pi/16, 1.414) circle (.9pt);
\draw[color=red] (25*pi/16, 0) circle (1.5pt);
\draw (25*pi/16,-.1) -- (25*pi/16,.1) node[below=1pt] {\tiny $(\pi+\sqrt{\mu})^2$};
  \draw[dotted][color=gray] (25*pi/16,0) -- (25*pi/16,1.414);
  
\draw[fill][color=blue] (16*pi/9, -1) circle (.9pt);
\draw[color=blue] (16*pi/9, 0) circle (2.5pt);
\draw (16*pi/9,-.1) -- (16*pi/9,.1) node[above=1pt] {\tiny $(\pi+\sqrt{\nu})^2$};
  \draw[dotted][color=gray] (16*pi/9,-1) -- (16*pi/9,0);
  
\draw[fill][color=blue] (49*pi/16, 1.414) circle (.9pt);
\draw[color=blue] (49*pi/16, 0) circle (2.5pt);
\draw (49*pi/16,-.1) -- (49*pi/16,.1) node[above=1pt] {\tiny $(2\pi-\sqrt{\mu})^2$};
 \draw[dotted][color=gray] (49*pi/16,0) -- (49*pi/16,1.414);

\draw[fill][color=red] (25*pi/9, -1) circle (.9pt);
\draw[fill][color=red] (pi,2) circle (.8pt);
\draw[color=red] (pi,0) circle (1.5pt);
\draw[color=red] (25*pi/9, 0) circle (1.5pt);
\draw (25*pi/9,-.1) -- (25*pi/9,.1) node[below=1pt] {\tiny $(2\pi-\sqrt{\nu})^2$};
 \draw[dotted][color=gray] (25*pi/9,-1) -- (25*pi/9,0);
\end{tikzpicture}\\[5pt]

\begin{tikzpicture}[domain=0:10]
    \draw[->] (-0.2,0) -- (11.2,0) node[right] {$\lambda$};
    \draw[->] (0,-2.4) -- (0,2.4);
    \draw[dotted][color=gray] (0, -1.414) -- (10.2, -1.414);
    \draw[dotted][color=gray] (0, 1.414) -- (10.2, 1.414);
       \draw[dotted][color=gray](0,2) -- (10.2,2);
       \draw[dotted][color=gray] (0,-2) -- (10.2,-2);
    \draw[color=red] plot[id=-cossqrt] function{-2*cos(sqrt(x*pi))} 
        node[right] {$-\cos\sqrt{\lambda}$};
    \draw[color=blue] plot[id=cossqrt] function{2*cos(sqrt(x*pi))} 
        node[right] {$\cos\sqrt{\lambda}$};
\draw (0,2) -- (0,2) node[left=3pt] {$1$};
\draw (0,-2) -- (0,-2) node[left=3pt] {$-1$};
\draw (0,1.414) -- (0,1.414) node[left=3pt] {$\alpha$};
\draw (0,1.2) -- (0,1.2) node[left=3pt] {$\vdots$};
\draw (0,-1) -- (0,-1) node[left=3pt] {$\vdots$};
\draw (0,-1.414) -- (0,-1.414) node[left=3pt] {$-\alpha$};
\draw (pi ,-.1) -- (pi ,.1) node[below=3pt] {\tiny $\pi^2$};
 \draw[dotted][color=gray] (pi,-2) -- (pi,2);
\draw (0,-.1) -- (0,.1) node[below left=3pt] { $0$};
\draw (pi/16,-.1) -- (pi/16,.1) node[below=3pt] {\tiny $\mu$};
\draw (9*pi/16,-.1) -- (9*pi/16,.1) node[below=3pt] {\tiny $(\pi-\sqrt{\mu})^2$};
\draw (25*pi/16,-.1) -- (25*pi/16,.1) node[below=3pt] {\tiny $(\pi+\sqrt{\mu})^2$};
\draw (49*pi/16,-.1) -- (49*pi/16,.1) node[below=3pt] {\tiny $(2\pi-\sqrt{\mu})^2$};

 \draw[dotted][color=gray] (pi/16,-1.414) -- (pi/16,1.414);
  \draw[dotted][color=gray] (9*pi/16,-1.414) -- (9*pi/16,1.414);
  \draw[dotted][color=gray] (25*pi/16,-1.414) -- (25*pi/16,1.414);
  \draw[dotted][color=gray] (49*pi/16,-1.414) -- (49*pi/16,1.414);

\draw[fill][color=black] (0, 2) circle (1.7pt);
\draw[fill][color=blue] (0, 2) circle (.8pt);
\draw[color=blue] (0,0) circle (2.5pt);

\draw[fill][color=black] (0, -2) circle (1.7pt);
\draw[fill][color=red] (0,-2) circle (.8pt);
\draw[color=red] (0,0) circle (1.5pt);

\draw[fill][color=black] (0, 1.414) circle (1.5pt);
\draw[fill][color=blue] (pi/16, 1.414) circle (.9pt);
\draw[color=blue] (pi/16, 0) circle (2.5pt);

\draw[fill][color=black] (0, -1.414) circle (1.5pt);
\draw[fill][color=red] (pi/16, -1.414) circle (.9pt);
\draw[color=red] (pi/16, 0) circle (1.5pt);

\draw[fill][color=red] (9*pi/16, 1.414) circle (.9pt);
\draw[color=red] (9*pi/16, 0) circle (1.5pt);

\draw[fill][color=blue] (9*pi/16, -1.414) circle (.9pt);
\draw[color=blue] (9*pi/16, 0) circle (2.5pt);

\draw[fill][color=red] (25*pi/16, 1.414) circle (.9pt);
\draw[color=red] (25*pi/16, 0) circle (1.5pt);

\draw[fill][color=blue] (25*pi/16, -1.414) circle (.9pt);
\draw[color=blue] (25*pi/16, 0) circle (2.5pt);

\draw[fill][color=blue] (49*pi/16, 1.414) circle (.9pt);
\draw[color=red] (49*pi/16, 0) circle (1.5pt);

\draw[fill][color=red] (49*pi/16, -1.414) circle (.9pt);
\draw[color=blue] (49*pi/16, 0) circle (2.5pt);

\draw[fill][color=red] (pi,2) circle (.8pt);
\draw[color=red] (pi,0) circle (1.5pt);

\draw[fill][color=blue] (pi,-2) circle (.8pt);
\draw[color=blue] (pi,0) circle (2.5pt);
\end{tikzpicture}

\caption{On the abscissa, the eigenvalues of $-\Delta^{CK}$ and $-\Delta^{KC}$ are plotted in blue and red, respectively, in correspondence with the associated eigenvalues of $\mathcal Z$ on the ordinate axis, which are plotted in black. The first plot reflects the case of a non-bipartite graph, for which the spectrum of $\mathcal Z$ is not symmetric with respect to 0; while the latter corresponds to the bipartite case.}
\end{center}
\end{figure}
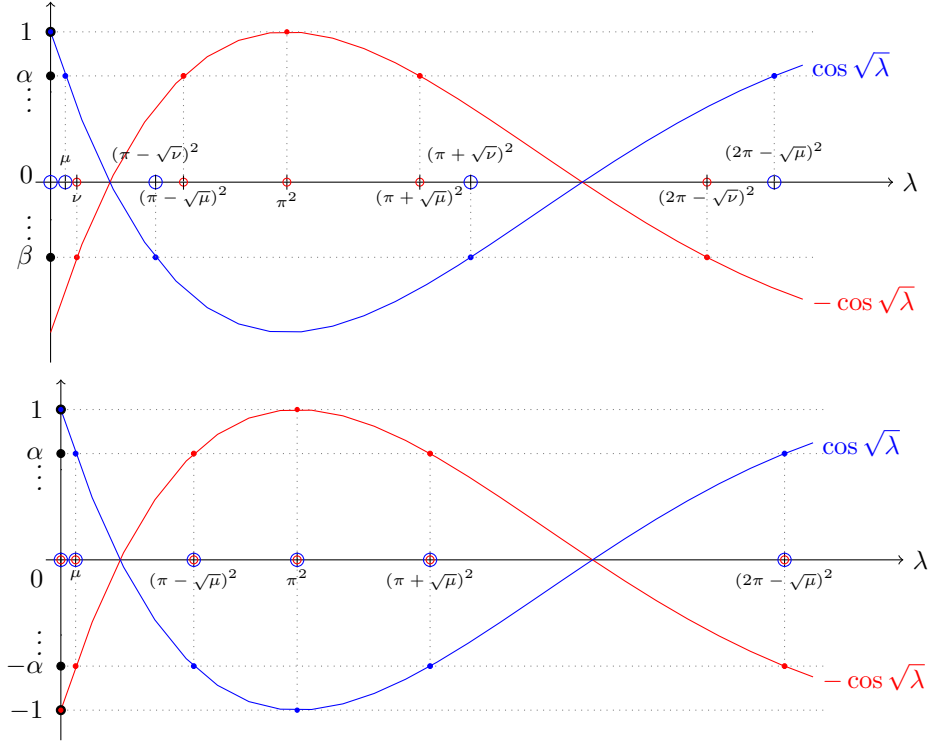

As eigenvalues corresponding
to $\pm \cos\sqrt \l\in\sigma(\ZD)$ are counted in both cases in the same order, the bipartite case
permits the following comparison of the eigenvalues. For trees, we find
$$\forall k\in\NN:\quad\lambda_k^{CK}\leq \lambda_k^{K C}.
$$
If $\Gamma$ is bipartite and unicyclic, i.e. $N=n$ with even circuit, then
$$\forall k\in\NN:\quad\lambda_k^{CK}= \lambda_k^{K C}.
$$
If $\Gamma$ is bipartite with general corank $\geq 2$ , then
$$\forall k\in\NN:\quad\lambda_k^{CK}\geq \lambda_{k+N-n}^{K C} \geq \lambda_{k}^{K C}. 
$$\end{rem}

However, no general uniform comparison seems to be available in the non-bipartite case.

\begin{exa}
\label{complete}
Let $K_n$ denote the complete graph on $n\geq 3$ vertices. Then the eigenvalues in the $CK$--case read
\begin{center}
 \begin{tabular}{|l|}
  \hline
   \\
    {\bf (1)} $\l =0,\; m(0)=1,\;\f=\be $\\
    \\
    \hline
    \\
    {\bf (2)} $\l =4\p^2  k^2,\,k\ne 0,\;m(\l)=2+\frac{1}{2}n(n-3),\;\f=\be $\\
  \\
  \hline
  \\
    {\bf (3)} $ \l =\p^2(2k+1)^2,\,\;m(\l)=\frac{1}{2}n(n-3),\;\f={\bf 0}$\\
  \\
  \hline
  \\
   \textcolor{black}{{\bf (4)} $ \cos  \sqrt{\l}=-(n-1)^{-1}, \quad m(\l)=n-1,\;\f\in\hbox{ker}\,\be\be^* $}\\
   \\
  \hline
    \end{tabular}\\
\end{center}
while in the $K C$--case we obtain the following ones.
\begin{center}
 \begin{tabular}{|l|}
  \hline
   \\
    {\bf (1)} $\l =0,\; m(0)=\frac{1}{2}n(n-3),\;\psi=0 $\\
    \\
    \hline
    \\
    {\bf (2)} $\l =4\p^2  k^2,\,k\ne 0,\;m(\l)=\frac{1}{2}n(n-3),\;\psi=0 $\\
  \\
  \hline
  \\
    {\bf (3)} $ \l =\p^2(2k+1)^2,\,\;m(\l)=2+\frac{1}{2}n(n-3),\;\psi=\be  $\\
  \\
  \hline
  \\
   \textcolor{black}{{\bf (4)} $ \cos  \sqrt{\l}=(n-1)^{-1}, \quad m(\l)=n-1,\;\psi\in\hbox{ker}\,\be\be^* $}\\
   \\
  \hline
    \end{tabular}\\
\end{center}
There is no uniform inequality between the $\lambda_k^{CK}$ and the $ \lambda_{k}^{K C}$. For $n=3$ e.g., $\lambda_0^{CK}< \lambda_{0}^{K C}$, while $\lambda_1^{CK}> \lambda_{1}^{K C}$.
\end{exa}

\begin{rem}
One of the most distinguished features of the characteristic equations for $-\Delta^{CK}$ and $-\Delta^{KC}$ is the r\^ole played by $\mathcal Z$. This becomes even more apparent once we recall that the spectrum of the transition matrix $\ZD$ of a graph essentially agrees (up to a reflection in the point $\frac12$) with the spectrum of the so-called normalized Laplacian $\mathcal L$ defined by
$$\mathcal{L}:=D^{-\frac12} L \,D^{-\frac12}=I-
D^{-\frac12}{\mathcal A} \,D^{-\frac12}= I-
D^{-\frac12}D\ZD \,D^{-\frac12}=D^\frac12\left(I-\ZD\right) D^{-\frac12} ,$$
where we have set
$$L:=\hbox{\rm Diag}\left(\AD\,\be\right)-\AD\quad\text{and}\quad D:=\hbox{\rm Diag}\left(\AD\,\be\right).
$$

Thus, the eigenvalues of $\ZD$ are precisely those of $I-\mathcal{L}$ counting multiplicities, in particular
\begin{equation}\label{transnormlapl}
\lambda \in \sigma (\mathcal Z) \Leftrightarrow 1-\lambda \in \sigma (\mathcal L).
\end{equation}
This correspondence has been routinely exploited to study asymptotics of random walks on graphs, see e.g.~\cite[\S~1.5]{Chu97}, in particular because the matrix $\mathcal L$ has been studied very thoroughly over the last two decades. For our purposes, it is particularly interesting that applying known spectral comparison results for graph operations together with the strong monotony properties of the functions $[0,\pi^2]\ni \lambda\mapsto \pm \cos\sqrt{\lambda}\in [-1,1]$, we can obtain some spectral monotony properties for $-\Delta^{CK}$ and $-\Delta^{KC}$ as well.
\end{rem}

Hence, some comparison results about the spectra of $-\Delta^{CK}$ and $-\Delta^{KC}$ become easily amenable, simply exploiting basic information about the spectra of $\mathcal L$. A curious examples is given in the following, where a special r\^ole is played by a strong kind of symmetry. We recall that a graph is said to be \emph{distance transitive} if for any two pairs of vertices $(x,y),(v,w)$ such that the distances between $x,y$ and between $v,w$ agree there is a graph automorphism mapping $x$ to $v$ and $y$ to $w$, cf.~\cite[\S~20]{big}. E.g., complete graphs, cube graphs, and the Petersen graph have this property.
\begin{prop}
Assume $\Gamma$ to be a distance transitive graph of diameter $d$. Then, both $-\Delta^{CK}$ and $-\Delta^{KC}$ have exactly $2(d+1)$ distinct eigenvalues in each interval $((2k\pi)^2,(2(k+1)\pi)^2]$, $k\in \mathbb N$.
\end{prop}
\begin{proof}
Taking into account  Theorems~\ref{mgeoalg} and~\ref{mak}, the assertion is a direct consequence of~\cite[Thm.~7.10]{Chu97} (or of~\cite[Prop.~21.2]{big}, after observing that vertex transitive graphs are necessarily regular).
\end{proof}
\begin{rem}
If one allows $\mathcal Z$ and hence $-\Delta^{CK},-\Delta^{KC}$ to be weighted, stronger assertions can be proved leading to spectral comparisons between Laplacians on distance transitive $\Gamma$ and on some related path graphs, cf.~\cite[Thms.~7.11 and~7.13]{Chu97}. We will not further elaborate on this point.
\end{rem}

\begin{rem}
Producing results analogous to that in Proposition~\ref{compar-ck} for $-\Delta^{KC}$ is less obvious, since less information seems to be available about the \emph{smallest} eigenvalue of $\mathcal Z$, i.e., about the \emph{largest} eigenvalue of $\mathcal L$ (apart from the bipartite case, of course). However, in the last few years many interlacing results have been obtained for the eigenvalues of $\mathcal L$ under graph operations. For example, it is proved in~\cite[Thm.~1.2]{But07} that if
\begin{itemize}
\item $\Gamma_1$ is a graph on $n$ vertices,
\item $\tilde{\Gamma}_2$ is a subgraph of $\Gamma_1$, 
\item $t$ is the number of isolated vertices of $\tilde{\Gamma}_2$,
\item $\Gamma_2$ is the complement of $\tilde{\Gamma}_2$ in $\Gamma_1$, i.e., the graph obtained by deleting from $\Gamma_1$ the edges of $\tilde{\Gamma}_2$,
\item $\mu_j^{(1)}$ and $\mu_j^{(2)}$ denote  the eigenvalues of the normalized Laplacian of $\Gamma_1$, $\Gamma_2$, respectively,
\end{itemize}
 then for $k=1,\ldots,n$ 
\[
\mu^{(1)}_{k-t+1}\le \mu^{(2)}_{k}\le 
\left\{
\begin{array}{ll}
\mu^{(1)}_{k+t-1},\qquad &\hbox{if $\Gamma_2$ is bipartite,}\\
\mu^{(1)}_{k+t},\qquad &\hbox{otherwise,}\\
\end{array} 
\right.
\]
where $\mu^{(1)}_{-t+1}=\ldots=\mu^{(1)}_0:=0$ and $\mu^{(1)}_{n+1}=\ldots=\mu^{(1)}_{n+t}:=2$. (An analogous relation holds if a graph is added, instead of subtracted~\cite[Cor.~1.4]{But07}). Yet more refined results related to more subtle structures (like coverings and spanning subgraphs) have been recently obtained in the doctoral thesis of Butler~\cite{But08} and in still unpublished lecture notes by Hall~\cite[\S~3]{Hal10}.  In view of Theorems~\ref{mgeoalg} and~\ref{mak}, it is clear how to translate all these interlacing results for the spectrum of $\mathcal L$ into  interlacing results for spectral subsets of $-\Delta^{CK}$ or $-\Delta^{KC}$.  
\end{rem}

\section{Inverse spectral aspects}\label{sec:inverse} 

Just like under the $CK$-condition considered in Section \ref{adjcal}, one cannot recover the graph from  the eigenvalues of $-\Delta^{KC}$. In fact,
one can use the same couple of regular isospectral, but non-isomorphic graphs as in \cite{be6} based on
a well--known example by Hoffman and Ray--Chaudhuri, displayed by the left pair
in Fig.\ref{raych}. The right pair displays the same graphs where in each of both of them the white vertices have to be identified.
They cannot be neither isometric nor isomorphic, since the left graph is planar,
while the right one is not.
\begin{figure}[h]\label{raych}
\centering
\includegraphics[width=0.5\textwidth]{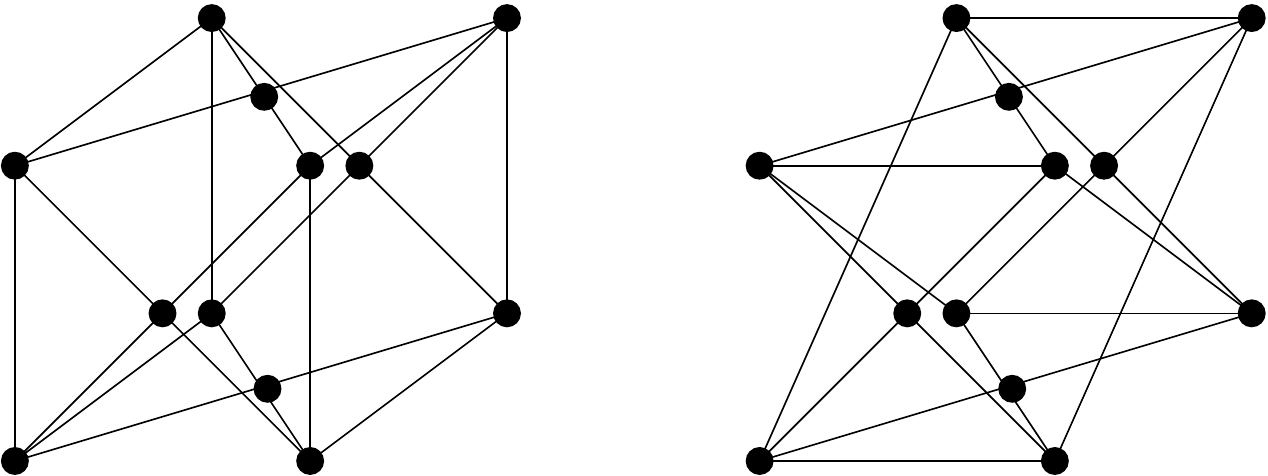}\hskip20mm \includegraphics[width=0.25\textwidth]{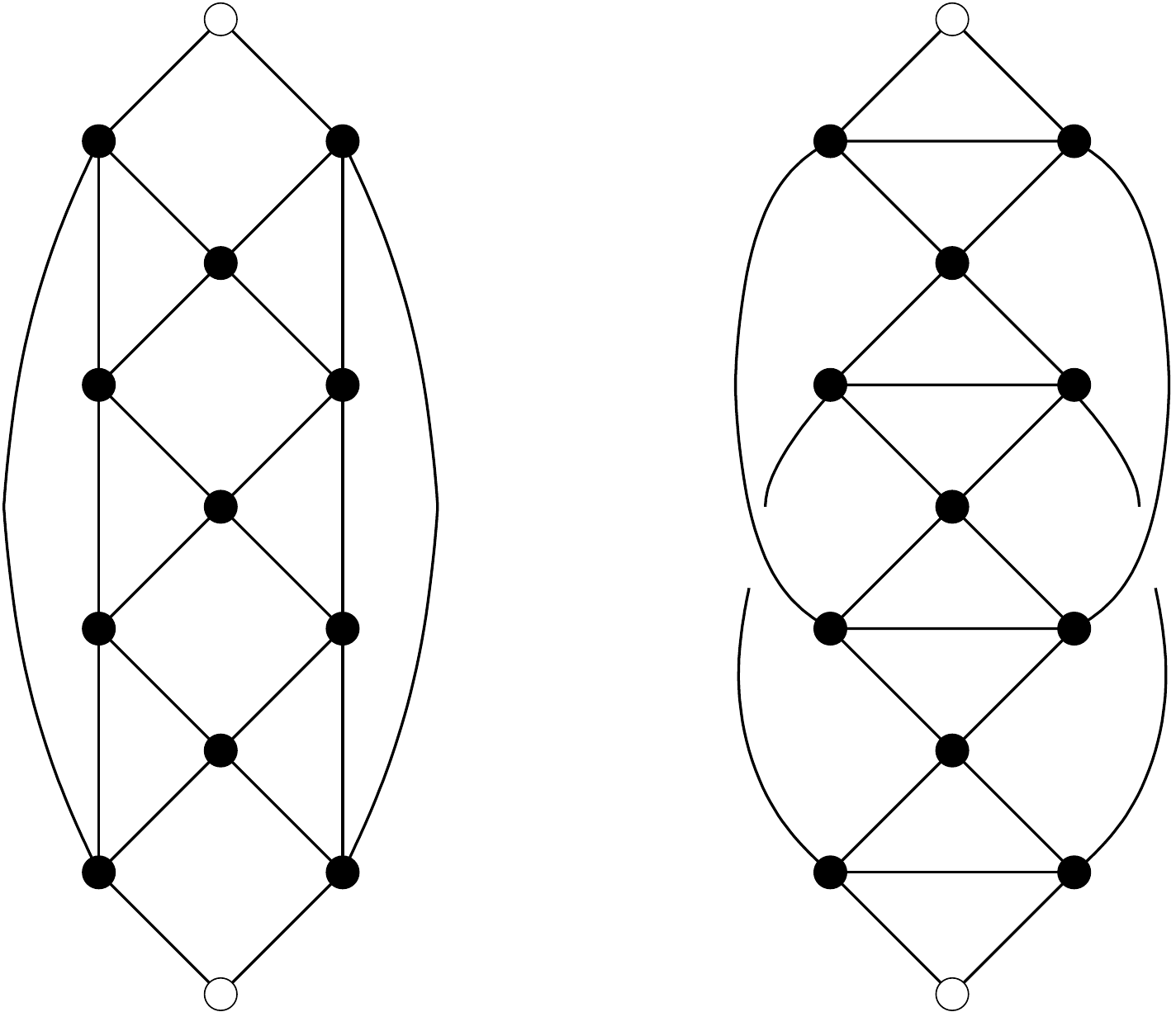}
\caption{Two non-isomorphic, but isospectral graphs under the $KC$--condition.}
\end{figure}
\begin{cor}
\label{onecannot} There are non-isomorphic pairs of regular graphs having the
same point spectrum of the Laplacian $-\Delta$ under
the $KC$--condition.
\end{cor}
In fact, the pair in Fig. \ref{raych} is not an exception. The above multiplicity formulae and the proof of Theorem \ref{recover}
show also that the notions of isospectrality with respect the Laplacian under the $CK$ and $KC$ condition, which we denote by $-\Delta^{CK}$ and $-\Delta^{KC}$ throughout, are equivalent.
\begin{cor}
\label{isospequ} Two graphs $\Gamma_1$ and $\Gamma_2$ are isospectral with respect to $-\Delta^{CK}$  if  and only if they are isospectral with respect to $-\Delta^{KC}$.
\end{cor}

\begin{rem}
>From the eigenvalues of $\ZD$ one can recover the number of vertices $n$ and the bipartiteness,
but, in general, the number of edges cannot be determined with the aid of $\sigma(\ZD)$, see \cite{Bel85,bugr}.
A very simple pair of graphs that are isospectral with respect to $\ZD$, but differ in the number of edges
is given by the circuit $C_4$ and the claw $K_{1,3}$ (star with $3$ edges) that have both the eigenvalues
$-2,0,0,2$.
By Theorems \ref{mgeoalg} and \ref{mak}, if two graphs are isospectral with respect to $-\Delta$ under the $CK$--condition or under the $KC$--condition, then 
they are also isospectral with respect to $\ZD$. But the latter condition is weaker than the first one, since one can recover
the number of edges from  the $-\Delta$--spectrum, see Theorem \ref{recover}. Note that the pair $C_4,K_{1,3}$ is neither isospectral
with respect to $-\Delta^{CK}$, nor with respect to $-\Delta^{KC}$, since the coranks are different.
\end{rem}

In the bipartite case, two unicyclic graphs are isospectral with respect to $-\Delta^{CK}$ if and only if they are isospectral with respect to $-\Delta^{KC}$, by Corollary~\ref{ckequalkc}.
But, one cannot determine the graph by means of these eigenvalues in this class of graphs.
\begin{cor} There are non-isomorphic pairs of bipartite unicyclic graphs that are  isospectral with respect to $-\Delta$ under the $CK$--condition or under the $KC$--condition.
\end{cor}
Indeed, the following example by S.\ Butler and J.\ Grout \cite{bugr} displays the desired properties. Let $\Gamma_1$ be the circuit of length $8$ and $\Gamma_2$
the circuit of length $4$ with two $2$--paths glued to one vertex as in  Figure~3. Then the eigenvalues of both $\ZD(\Gamma_1)$ and $\ZD(\Gamma_2)$ read
$$1, \frac{\sqrt{2}}{2},\frac{\sqrt{2}}{2}, 0, 0, -\frac{\sqrt{2}}{2},-\frac{\sqrt{2}}{2},-1.
$$
\begin{figure}[h]\label{exbugr}
\centering
\includegraphics[width=0.25\textwidth]{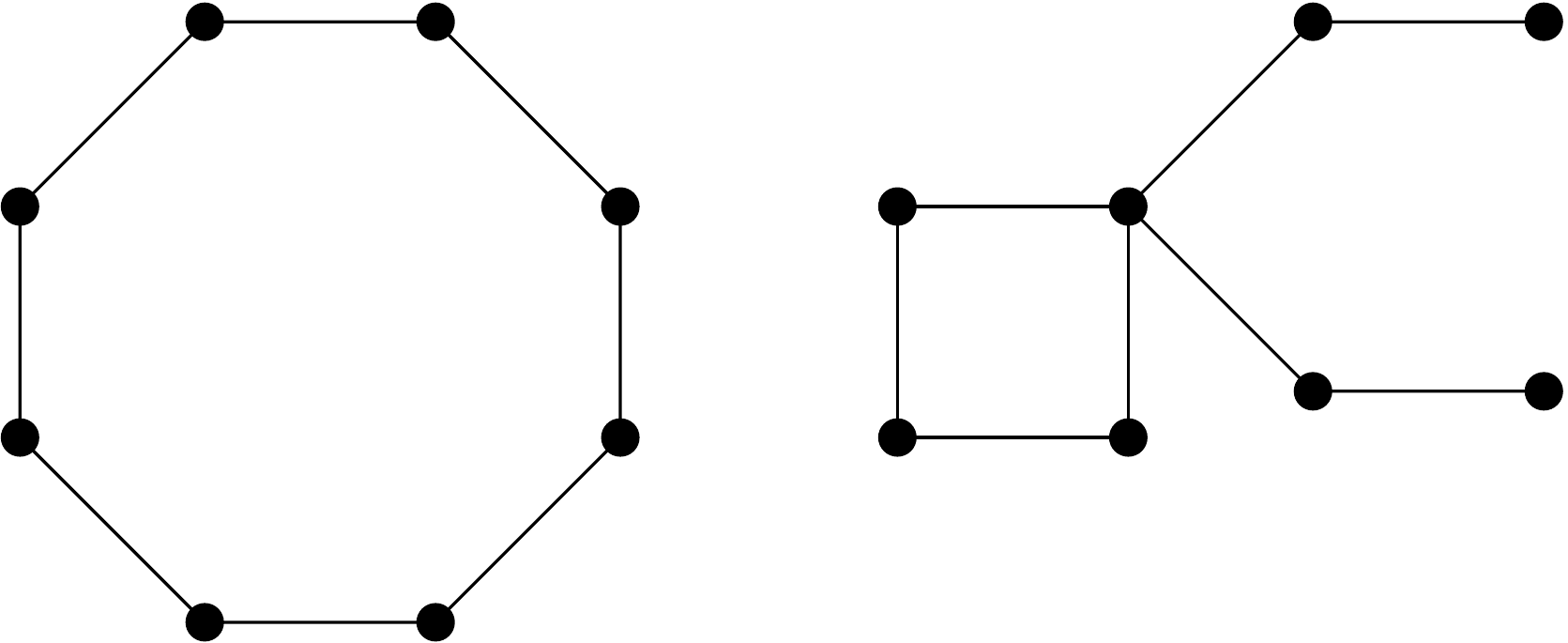}
\caption{The example from S.\ Butler and J.\ Grout \cite{bugr}: two non isomorphic bipartite and unicyclic, but isospectral graphs under the $CK$-- and the $KC$--condition.}
\end{figure}
The example also shows that one cannot recover the girth of the graph from the eigenvalues in question.
Combining Corollaries \ref{ckequalkc} and \ref{isospequ} yields that there cannot be a pair of non-isomorphic and non-bipartite unicyclic graphs that are isospectral with respect to $-\Delta^{CK}$ and $-\Delta^{KC}$. In other words 
\begin{cor} \label{recoverunic} If two unicyclic graphs are isospectral with respect to $-\Delta^{CK}$ or $-\Delta^{KC}$,
and at least one of them is not bipartite, then they are isomorphic.
\end{cor}
Recall that the \textit{complexity} $\kappa(\Gamma)$ of a graph $\Gamma$ is defined by the number of spanning trees in $\Gamma$.
By Kirchhoff's matrix--tree theorem, it can be calculated from the eigenvalues of
the combinatorial Laplacian $L=D-\AD.$
But this is not possible from the spectrum of the matrix $\ZD$. In fact, a formula due to Runge and Sachs shows that knowing the determinant of $\hbox{\rm Diag}\left(\AD\,\be\right)$ is necessary and sufficient in order to reconstruct the complexity of $\Gamma$ from the spectrum of $\mathcal Z$, cf.~\cite[\S~1.9, \#10]{cds}.

\begin{theo} \label{recover} From the eigenvalues (and their multiplicities) of either $-\Delta^{CK}$ or $-\Delta^{KC}$, the following invariants can be determined:
\begin{enumerate}
	\item the number of vertices $n$,
	\item the number $c(\Gamma)$ of connected components of the graph,
	\item the number of edges $N$,
	\item the numbers $c^+(\Gamma)$ and $c^-(\Gamma)$ of bipartite and non-bipartite components (and hence bipartiteness).
\end{enumerate}

However, it not possible to deduce from those eigenvalues the complexity, the degrees, and, in particular, whether the graph is regular.
\end{theo}
\begin{proof} 
Under the $CK$-condition, we first deduce $c(\Gamma)=m_g(0)$ and
\[
c^-(\Gamma)=\frac12 (m_g(4\pi^2)- m_g(\pi^2))\qquad \hbox{and}\qquad c^+(\Gamma)=c(\Gamma)-c^-(\Gamma).
\]
Under the $KC$-condition, instead, we deduce
\[
c^-(\Gamma)=\frac12 (m_g(\pi^2)-m_g(4\pi^2))\qquad \hbox{and}\qquad c^+(\Gamma)=m_g(4\pi^2)-m_g(0),
\]
and hence also $c(\Gamma)$. 
Since $\ZD$ is diagonalizable \cite{Bel85}, and since $m_g(1;\ZD)=c(\Gamma)$ and $m_g(-1;\ZD)=c^+(\Gamma)$ (recall that $\Gamma$ is bipartite if and only if $-1\in\sigma(\ZD)$), we deduce in both cases
\begin{equation}
\label{formel:n}
n=\sum_{\mu\in\sigma(\ZD)} m_g(\mu;\ZD)=c(\Gamma)+c^+(\Gamma) + \sum_{\lambda\in(0,\pi^2)} m_g(\lambda).
\end{equation}
Using Theorems~\ref{mgeoalg} and~\ref{mak} we obtain $N=n-2c(\Gamma)+m_g(4\pi^2)$ in the first case, and $N=m_g(0)+n-c^+(\Gamma)$ in the second one.

The pair of isospectral, but non-isomorphic graphs in  Figure~3 shows that it is impossible to determine the vertex
degrees, in particular whether they are all equal, from the data in question. Finally, the same pair displays complexities
$\kappa=8$ and $\kappa=4$, respectively.
\end{proof}

\begin{rem}
It is well-known that $n$, $N$, and $c(\Gamma)$ can be deduced from the adjacency matrix $\mathcal A$, or from a combination of both the combinatorial and the normalized Laplacian.
Moreover, it is known that $c^+(\Gamma)$, the number of bipartite components, agrees with the multiplicity of $0$ as an eigenvalue of the \emph{signless Laplacian} 
\[
{\mathcal Q}:=\hbox{\rm Diag}\left(\AD\,\be\right)+\AD,
\]
a matrix which has gained much attention in recent years, cf.~\cite{crs07}. To the best of our knowledge, however, it is unknown whether $c^+(\Gamma)$ can be found only knowing the spectrum of any of the other relevant matrices; and also whether $c(\Gamma)$ can be reconstructed from the knowledge of the spectrum of $\mathcal Q$. Hence, it seems that $\Delta^{CK}$ and $\Delta^{KC}$ are, taken individually, more comprehensive sources of spectral information than the usual matrices considered in spectral graph theory.
\end{rem}

In the regular case (say, with degree $\gamma$), the notions of isospectrality with respect to 
\begin{itemize}
\item the adjacency matrix $\mathcal A$, 
\item the combinatorial Laplacian $L=D-\AD=\gamma I-\AD$, 
\item the transition matrix ${\mathcal Z}:=D^{-1}\AD=\gamma^{-1}\AD$, 
\item the normalized (discrete) Laplacian $\mathcal{L}:=D^{-\frac12}LD^{-\frac12}=I-\gamma^{-1}\AD$, and 
\item the signless Laplacian $\mathcal{Q}:=D+\AD=\gamma I+\AD$,
\end{itemize}
where $D:=\hbox{\rm Diag}\left(\AD\,\be\right)$, are mutually equivalent. Thus, we are
led to the following    
\begin{cor} \label{regularrec} If $\Gamma$ is regular, then one can recover all the invariants mentioned in Theorem \ref{recover}, as well as the complexity from the eigenvalues of each of the operators or matrices $-\Delta^{KC}$, $-\Delta^{CK}$, $\mathcal A$, $L$, $\ZD$, $\mathcal{L}$, and  $\mathcal Q$.
\end{cor}

\begin{rem} Combining the multiplicities for harmonic functions, Theorems~\ref{mak} and ~\ref{mgeoalg}, yield
$${\rm dim}\; \Ker \Delta^{KC}-{\rm dim}\;\Ker \Delta^{CK}=N-n-c^-(\Gamma).
$$
In the cohomological language, $-\Delta^{KC}$ can be written as $dd^*$ and $-\Delta^{CK}$ as $d^*d$, where $d$ denotes the first derivative with continuity vertex conditions. It follows that
\[
{\rm ind }\; d:={\rm dim}\, \Ker \Delta^{KC}-{\rm dim}\, \Ker \Delta^{CK}
\]
agrees with the Euler characteristic ($N-n$) of $\Gamma$ if $\Gamma$ is bipartite ($c^-(\G)=0$). This yields an easy proof of ~\cite[Thm.~20]{FulKucWil07}
in the bipartite case, but
only then. The formula ${\rm ind }\; d=N-n$ does not hold in the non-bipartite case. 
This should be compared with the theory developed in \cite{Pos09}. The result of Fulling, Kuchment and
Wilson aims at comparing the null spaces of an operator acting on
$L^2(0,1;\ell^2(E))$--functions and another operator acting on the space of
their derivatives (i.e., on a space of 0-forms and 1-forms,
respectively), whereas we regard both operators as acting on the same
space. 
If the underlying graph is bipartite, then it is always possible to find
an orientation that allows to formulate KC-condition for a function u as
the CK-condition for its derivative, as done in ~\cite{FulKucWil07}; in the
non--bipartite case, this is in general impossible, as even the example
of a circuit of length 3 shows. This explains why our results are not
comparable with those of ~\cite{FulKucWil07} in the non-bipartite case. In particular,
their formula (36) and our Theorem 5.1 are only seemingly mutually
contradictory -- in fact, they refer to different objects.
\end{rem}


\end{document}